\documentclass[a4paper, pdftex]{amsart}
\AtBeginDocument{%
  \def\MR#1{}
}

\usepackage{pdflscape}

\usepackage{tikz, tikz-cd}
\usepackage{amsmath,amssymb}
\usepackage{amsthm}
\usepackage{mathrsfs}
\usepackage[shortlabels]{enumitem}
\usepackage{mathtools}
\usepackage[all,cmtip,2cell]{xy}
\usepackage{array}
\usepackage{comment}
\usepackage{braket}
\usepackage{stmaryrd}
\usepackage{url}
\usepackage{ascmac}
\usepackage{multirow}
\usepackage{hhline}
\usepackage{wasysym}
\usepackage{graphicx}
\usetikzlibrary{calc}

\usepackage[top=30truemm,bottom=30truemm,left=35truemm,right=35truemm]{geometry}

\usepackage{setspace}
\newcommand{\marginparstretch}{0.6}
\let\oldmarginpar\marginpar
\renewcommand\marginpar[1]{\-\oldmarginpar[\framebox{\setstretch{\marginparstretch}\begin{minipage}{\marginparwidth}{\raggedleft\tiny #1}\end{minipage}}]{\framebox{\setstretch{\marginparstretch}\begin{minipage}{\marginparwidth}{\raggedright\tiny #1}\end{minipage}}}}

\usepackage[top=30truemm,bottom=30truemm,left=35truemm,right=35truemm]{geometry}

\usepackage{fixme}
\fxsetup{status=draft}

\theoremstyle{plain}
\newtheorem{thm}{Theorem}[section]
\newtheorem{prop}[thm]{Proposition}
\newtheorem{lem}[thm]{Lemma}

\newtheorem{cor}[thm]{Corollary}
\newtheorem*{thm*}{Theorem}

\theoremstyle{definition}
\newtheorem{defi}[thm]{Definition}
\newtheorem{conj}[thm]{Conjecture}
\newtheorem{setup}[thm]{Setup}
\newtheorem*{NaC}{Notation and Convention}

\theoremstyle{remark}
\newtheorem{rem}[thm]{Remark}

\newtheorem{ex}[thm]{Example}

\numberwithin{equation}{section}

\newcommand{\Z}{\mathbb{Z}}
\newcommand{\Gm}{\mathbb{G}_{\mathrm{m}}}
\newcommand{\Q}{\mathbb{Q}}
\newcommand{\R}{\mathbb{R}}
\newcommand{\C}{\mathbb{C}}
\newcommand{\A}{\mathbb{A}}
\renewcommand{\P}{\mathbb{P}}

\newcommand{\V}{\mathbb{V}}

\DeclareMathOperator{\id}{id}

\DeclareMathOperator{\Sym}{\mathrm{Sym}}

\DeclareMathOperator{\Spec}{\mathrm{Spec}}

\DeclareMathOperator{\Proj}{\mathrm{Proj}}

\DeclareMathOperator{\OGr}{\mathrm{OGr}}

\DeclareMathOperator{\Spin}{\mathrm{Spin}}

\DeclareMathOperator{\Bl}{\mathrm{Bl}}

\DeclareMathOperator{\codim}{\mathrm{codim}}

\DeclareMathOperator{\Hom}{Hom}
\DeclareMathOperator{\End}{End}

\DeclareMathOperator{\Ext}{Ext}

\DeclareMathOperator{\refl}{ref}

\DeclareMathOperator{\Tot}{Tot}

\newcommand{\mcE}{\mathcal{E}}
\newcommand{\mcF}{\mathcal{F}}
\newcommand{\mcG}{\mathcal{G}}

\newcommand{\mcK}{\mathcal{K}}
\newcommand{\mcO}{\mathcal{O}}

\newcommand{\mcP}{\mathcal{P}}

\newcommand{\mcS}{\mathcal{S}}
\newcommand{\mcT}{\mathcal{T}}
\newcommand{\mcU}{\mathcal{U}}
\newcommand{\mcV}{\mathcal{V}}
\newcommand{\mcX}{\mathcal{X}}

\newcommand{\mcZ}{\mathcal{Z}}

\newcommand{\wX}{\widetilde{X}}

\DeclareMathOperator{\coh}{coh}
\DeclareMathOperator{\Qcoh}{Qcoh}
\DeclareMathOperator{\Dcoh}{Dcoh}
\DeclareMathOperator{\Dmod}{Dmod}
\DeclareMathOperator{\modu}{mod}

\DeclareMathOperator{\OG}{\mathrm{OG}}

\DeclareMathOperator{\RG}{\mathrm{R}\Gamma}

\DeclareMathOperator{\RHom}{\mathrm{RHom}}
\newcommand{\RR}{\mathbf{R}}
\newcommand{\LL}{\mathbf{L}}

\DeclareMathOperator{\Db}{\mathrm{D}^{\mathrm{b}}}
\DeclareMathOperator{\D}{\mathrm{D}}
\DeclareMathOperator{\Perf}{Perf}

\DeclareMathOperator{\GL}{\mathrm{GL}}

\renewcommand{\mod}{\mathrm{mod}}

\mathchardef\mhyphen="2D

\title[D-equivalence for the simple flop of type $D_4$]{Derived equivalence for the simple flop of type $D_4$ via tilting bundles}
\author[W.HARA]{Wahei Hara}

\address[W.Hara]{Kavli Institute for the Physics and Mathematics of the Universe (WPI), University of Tokyo, 5-1-5 Kashiwanoha, Kashiwa, 277-8583, Japan.}
\email{wahei.hara@ipmu.jp}

\date{\today}



\begin{document}
\maketitle

\begin{abstract}
The aim of this article is to discuss the derived equivalence problem for a local model of the simple flop of type $D_4$, which was found by Kanemitsu \cite{Kan22}.
First, tilting bundles on both sides of the flop are constructed, 
and then those tilting bundles are applied to prove the derived equivalence.
This derived equivalence for the flop deduces derived equivalences between general K3 surfaces of degree $12$.
%The method in this article also shows the existence of noncommutative crepant resolutions of the singularity that are derived equivalent to both sides of the flop.
The study of this example of a flop is very similar to the author's previous work \cite{Hara24} for the simple of flop of type $G_2^{\dagger}$,
but the construction and the analysis of tilting bundles become harder.
\end{abstract}

%\tableofcontents

\section{Introduction and main results}

\subsection{Background}

Minimal model theory is a central topic in algebraic geometry,  
providing a procedure to obtain a better model of a given variety in terms of the canonical class, called a minimal model.  
Within this theory, there is an important class of birational transformations known as \textit{flops}.  
A flop preserves the canonical classes of varieties, and therefore, two varieties connected by a flop are regarded as lying at the same level in the minimal model program.  
For instance, a single variety may admit several distinct minimal models, any two of which are connected by a sequence of flops \cite{Kaw08}.

Interestingly, it is conjectured that flops preserve not only the canonical classes, but also the homological information of varieties.  
More precisely, a flops is expected to induce an equivalence of derived categories. 

\begin{conj}[\cite{BO02}] \label{DK conj}
Let $X_+ \xrightarrow{f_+} Y \xleftarrow{f_-} X_-$ be a diagram of a flop between two smooth varieties $X_+$ and $X_-$.
Then $X_+$ and $X_-$ are D-equivalent, i.e.~there exists an exact equivalence $\Db(\coh X_+) \simeq \Db(\coh X_-)$ of triangulated categories.
\end{conj}

This conjecture motivates the study of derived categories in algebraic geometry and reveals deep connections among birational geometry, representation theory, and mathematical physics.  
The statement of the conjecture should be extended to varieties with terminal singularities,
but the derived category of a singular variety is a subtle object,
and no good formulation of the conjecture for singular varieties is known at this moment.

The above conjecture has been proven true in dimension three by Bridgeland.  
However, it remains widely open in dimensions greater than or equal to four.  
One of the difficulties in proving the conjecture in higher dimensions lies in the complexity of the geometry of higher dimensional flops.  
Additionally, no classification of smooth flops is known in higher dimensions.

From this reason, it is natural to find a reasonable class of flops and study those flops.
One such class is given by \textit{simple flops}, defined by Kanemitsu \cite{Kan22}.
A simple flop is a flop that connects two varieties through a single smooth blow-up and a single smooth blow-down. 
It serves as a natural generalization of the Atiyah flop, the most fundamental example of a flop. 
Simple flops are important not only for their simplicity and the wealth of examples they provide in birational geometry,
but also for their connections with Mukai pairs and associated projective Calabi-Yau manifolds.

The aim of this article is to discuss the derived equivalence for the canonical local model of the simple flop of type $D_4$.

\subsection{The simple flop of type $D_4$ and the main result} \label{sect: flop}
This section explains the geometry of the simple flop of type $D_4$ and then states the main result.
Consider the orthogonal Grassmannians $\OGr(3,8)$ and $\OGr(4,8)$.
It is known that $\OGr(4,8)$ has two connected components, which are denoted by $\OGr_+(4,8)$ and $\OGr_-(4,8)$,
and both are isomorphic to the six dimensional smooth quadric hypersurface $\Q^6 \subset \P^7$.
There are natural projections $p_{\pm} \colon \OGr(3,8) \to \OGr_{\pm}(4,8)$.
Let $\mcS_{\pm}$ be the restriction of rank $4$ universal subbundle to $\OGr_{\pm}(4,8)$.
Under an identification $\OGr_{\pm}(4,8) \simeq \Q^6$, $\mcS_{\pm}$ is isomorphic to a spinor bundle over $\Q^6$.
In addition, the projections $p_{\pm}$ give identifications
\[ 
\begin{tikzcd}
  \P_{\OGr_+(4,8)}(\mcS_+(2)) \arrow[d, "p_+"'] \arrow[r, equal] & \OGr(3,8) \arrow[r, equal] & \P_{\OGr_-(4,8)}(\mcS_-(2)) \arrow[d, "p_-"]  \\
 \OGr_+(4,8) && \OGr_-(4,8).
\end{tikzcd} \]
This diagram is called the roof of type $D_4$, since they are all rational homogeneous manifolds of Dynkin type $D_4$ \cite{Kan22}.
Associated to this roof, there is a flop
\[ \begin{tikzcd}
X_+ \coloneqq \Tot_{\OG_+(4,8)}(\mcS_+^{\vee}(-2)) \arrow[rr, dashrightarrow] \arrow[rd, "f_+"'] & & X_- \coloneqq \Tot_{\OG_-(4,8)}(\mcS_-^{\vee}(-2)) \arrow[dl, "f_-"] \\
& \Spec R, & 
\end{tikzcd} \]
which will be called the canonical local model of the simple flop of type $D_4$ (Definition~\ref{def model flop}).
The following is the main result of this article.

\begin{thm} \label{main thm}
There exists an exact equivalence of $R$-linear triangulated categories 
\[ \Phi \colon \Db(\coh X_+) \xrightarrow{\sim} \Db(\coh X_-) \]
such that $R(f_+)_* \simeq  R(f_-)_* \circ \Phi$.
\end{thm}

More precisely, it is shown in this article that the local Calabi-Yau $10$-fold $X_+$ (resp.~$X_-$) admit a \textit{tilting bundle} $\mcT_+$ (resp.~$\mcT_-$) that contains $\mcO_{X_{+}}$ (resp.~$\mcO_{X_-}$) as a direct summand,
and they satisfy $\End_{X_+}(\mcT_+) \simeq \End_{X_-}(\mcT_-)$.
Then the basic property of tilting bundles immediately implies Theorem~\ref{main thm}. 

\begin{rem}
Note that the existence of an equivalence $\Db(\coh X_+) \xrightarrow{\sim} \Db(\coh X_-)$
of $R$-linear triangulated categories was first proved by Xie \cite{Xie24} for an arbitrary model $X_+ \dashrightarrow X_-$ of the simple flop of type $D_4$.
His proof also extends to relative situations.
On the other hand, although our proof works only for the canonical local model (and the complete local model),
it has two advantages.
One of them is that our equivalence $\Phi$ satisfies $\Phi \circ R(f_-)_* \simeq R(f_+)_*$,
which is remarkable in the following sense. 
By definition of the flop, there is no isomorphism $\varphi \colon X_+ \xrightarrow{\sim} X_-$ such that $\varphi \circ f_+ \simeq f_-$, but the above functor isomorphism shows that such an identification exists for derived categories.
The other one is that our proof uses tilting bundles, 
and hence it also shows the existence of a noncommutative crepant resolution (c.f.~\cite{VdB04, VdB23})
that is derived equivalent to both crepant resolutions $X_{\pm}$ of $\Spec R$.
This aspect is important from the perspective of the generalised McKay correspondence (see \cite[Section~1.3.2]{Hara24}).
\end{rem}

\subsection{Application to K3 surfaces of degree $12$}
Although Theorem~\ref{main thm} conforms Conjecture~\ref{DK conj} for a single example of a flop only,
the constructed derived equivalence $\Phi$ is strong enough to deduce the result on the derived categories of general K3 surfaces of degree $12$ in a uniform way.

\begin{cor} \label{main cor}
Let $s_+ \in H^0(\mcS_+(2))$ be a section, and $s_- \in H^0(\mcS_-(2))$ the corresponding section.
If $s_+$ and $s_-$ are regular, then the equivalence in Theorem~\ref{main thm} deduces an equivalence of derived categories
\[ \Db(\coh V(s_+)) \simeq \Db(\coh V(s_-)) \]
of (possibly singular) K3 surfaces.
\end{cor}

Such a class of K3 surfaces has already been studied in many literatures including \cite{Mukai88, Kuz06, IMOU20, KR22} from various perspectives,
the result in this article adds one more aspect from birational geometry and derived category.
See Remark~\ref{remark for K3} for further details.

\subsection{Comparison with previous works and connections}
Some of key geometric facts used in this article to construct tilting bundles have already appeared in the author's previous work \cite{Hara24}, which studies the simple flop of type $G_2^{\dagger}$.
This is not surprising, as the geometry of the simple flop of type $D_4$ is very close to that of type $G_2^{\dagger}$.
However, the result in \cite{Hara24} does not imply the results in this article, 
and the construction of tilting bundles for the type $D_4$ flop is more involved than for type $G_2^{\dagger}$, 
as the tilting bundles here have more indecomposable direct summands.
In addition, the algebraic group that is used to study the type $G_2^{\dagger}$ flop was $\Spin(7)$, 
whose Dynkin type is $B_3$, 
but in this article the representation theory of the algebraic group $\Spin(8)$ of Dynkin type $D_4$ is required.
%This is another difference between this article and \cite{Hara24}.
Furthermore, to prove certain key vanishings of $\operatorname{Ext}$ groups, we apply several results from local cohomology theory.

Combining the result in this article with those of other works \cite{Seg16, Hara21, Hara24, DHKR25}, derived equivalences via tilting bundles have now been constructed for the smallest-rank examples in all known Dynkin types, except for $F_4$.
Appendix~\ref{appendix} summarises the existing works on simple flops and their derived categories.
See also \cite[Section~1.3.3]{Hara24}.

At the time of writing, 
it sounds difficult to generalise the results in this article to simple flops of type $D_n$ for general $n > 4$.
For example, our construction of tilting bundles for type $D_4$ flop uses full strong exceptional collections of vector bundles over $\OG_{\pm}(4,8)$.
However, to the best of the author's knowledge, 
the existence of full exceptional collection of vector bundles over $\OG_{\pm}(n,2n)$ has not been known if $n \geq 7$.

\begin{NaC}
This article works over the complex number field $\C$.
This article adopts the following notations.
\begin{enumerate}
\item[$\bullet$] $\Q^n$ : the $n$-dimensional smooth quadric hypersurface of $\P^{n+1}$.
\item[$\bullet$] $\OG(m,n)$ : the orthogonal Grassmannian of $m$-dimensional linear subspaces of $\C^n$.
\item[$\bullet$] $\P(\mcE) \coloneqq \Proj \Sym \mcE$.
\item[$\bullet$] $\Tot(\mcE) \coloneqq \V(\mcE^{\vee}) \coloneqq \Spec \Sym \mcE^{\vee}$
\item[$\bullet$] $\mcF_{\omega}$ : the irreducible homogeneous vector bundle corresponding to a weight $\omega$.
\item[$\bullet$] $\refl(X)$ : the category of reflexive sheaves over $X$.
\end{enumerate}
\end{NaC}

\subsection*{Acknowledgements}
The author would like to thank Yuki Hirano, Marco Rampazzo and Ying Xie for very helpful conversations.
This work was supported by World Premier International Research Center Initiative (WPI), MEXT, Japan, 
and by JSPS KAKENHI Grant Number JP24K22829.

\section{Preliminaries}

\subsection{Tiling bundles}

This section will summarise the definitions and basic properties of tilting bundles and noncommutative crepant resolutions.
During this section all schemes are assumed to be Noetherian, separated, and of finite Krull dimension, 
although for some definitions and propositions these properties do not give a minimal assumption for schemes.

\begin{defi}
Let $Z$ be a scheme, and $\mcT \in \Perf(Z)$ a perfect complex.
\begin{enumerate}
\item[(1)] $\mcT$ is called \textit{pretilting} if $\Ext_Z^i(\mcT, \mcT) = 0$ for all $i \neq 0$.
\item[(2)] $\mcT$ is called \textit{tilting} if $\mcT$ is pretilting and $\mcT$ generates the unbounded derived category $\D(\Qcoh Z)$ of quasi-coherent sheaves.
\end{enumerate}
\end{defi}

\begin{defi}
Let $R$ be a normal Gorenstein domain, and $M \in \refl(R)$ a reflexive $R$-module.
We say that $M$ gives a noncommutative crepant resolution (= NCCR)
if the following two conditions are satisfied.
\begin{enumerate}
\item[(1)] The algebra $\End_R(M)$ has finite global dimension, and
\item[(2)] $\End_R(M)$ is Cohen-Macaulay as an $R$-module.
\end{enumerate}
The algebra $\End_R(M)$ is called an \textit{NCCR} of $R$.
\end{defi}

A connection between two concepts is given by the following well-known proposition.

\begin{prop}
Let $Z$ be a scheme that is projective over an affine scheme $\Spec R$.
Assume that $Z$ admits a tilting bundle $\mcT$.
Then the following holds.
\begin{enumerate}
\item[\rm (1)] The functor
\[ \RHom_Z(\mcT, -) \colon \Db(\coh Z) \to \Db(\modu \End_{Z}(\mcT)) \]
is an equivalence of $R$-linear triangulated categories.
\item[\rm (2)] If the structure morphism $f \colon Z \to \Spec R$ is a crepant resolution, then
\begin{enumerate}
\item[\rm (a)] There are isomorphisms of $R$-algebras
\[ \End_Z(\mcT) \simeq \End_R(f_*\mcT) \simeq \End_R((f_*\mcT)^{\vee\vee}), \]
where $(-)^{\vee}$ denotes the dualising functor $\Hom_R(-,R)$.
\item[\rm (b)] The reflexive $R$-module $(f_*\mcT)^{\vee\vee}$ gives an NCCR $\End_Z(\mcT)$.
\item[\rm (c)] If $\mcT$ contains $\mcO$ as a direct summand, then $f_*\mcT$ is a Cohen-Macaulay $R$-module (and hence reflexive).
\end{enumerate}
\end{enumerate}
\end{prop}

\begin{proof}
The proof for (1) can be found in \cite[Lemma~3.3]{TU10}.
For (2), (a) and (b) are shown in \cite[Lemma~2.21 and Theorem~2.22]{HaraHirano24},
and (c) is proved in \cite[Lemma~A.2]{TU10}.
\end{proof}

In order to check the partial tilting property of a given bundle, 
the following result from local cohomology theory is often useful.

\begin{lem} \label{loc coho prop}
Let $Z$ be a Cohen-Macaulay scheme and $j \colon U \to Z$ an open immersion.
Assume that the complement $W = Z \setminus j(U)$ satisfies $\codim_Z W = c$.
Then for any locally free sheaf $\mcE$ of finite rank, the natural map
\[ H^i(Z, \mcE) \to H^i(U, j^*\mcE) \]
is an isomorphism for all $i \leq c-2$.
\end{lem}

\begin{proof}
This follows from a standard argument in local cohomology theory.
See \cite[Corollary~2.13]{Hara20}, for example.
\end{proof}

\subsection{Simple flops}

The notions of simple flops and simple K-equivalences were introduced by Kanemitsu \cite{Kan22}.

\begin{defi}
\begin{enumerate}
\item[(1)] A \textit{Mukai pair} $(Z, \mcE)$ is a pair of a Fano manifold $Z$ and an ample vector bundle $\mcE$ such that $c_1(\mcE) = c_1(V)$.
\item[(2)] A Mukai pair $(Z, \mcE)$ is called \textit{simple} if $Z$ is of Picard rank one and $\P_Z(\mcE)$ admits another $\P^{r-1}$-bundle structure, where $r$ is the rank of $\mcE$.
\item[(3)] The projectivization $\P_Z(\mcE)$ associated to a simple Mukai pair $(Z,\mcE)$ is called a \textit{roof}.
\end{enumerate}
\end{defi}

\begin{ex} \label{rem roof K3}
Consider the two connected components $\OG_{\pm}(4,8)$ and their universal subbundles $\mcS_{\pm}$ for each $+$ and $-$.
Under an identification $\OG_{\pm}(4,8) \simeq \Q^6$, those bundles are one of spinor bundles over $\Q^6$.
Then as expressed in Section~\ref{sect: flop}, $(\OG_{\pm}(4,8), \mcS_{\pm}(2))$ are simple Mukai pairs, and they give the same roof $\OG(3,8)$ \cite[Example~5.5]{Kan22}.
\end{ex}

\begin{defi}
A birational map $\mcX_+ \dashrightarrow \mcX_-$ between two smooth varieties $\mcX_{\pm}$ is called a \textit{simple K-equivalence} if the following conditions are satisfied.
\begin{enumerate}
\item[(1)] The birational map is resolved as $\mcX_+ \xleftarrow{\phi_+} \widetilde{\mcX} \xrightarrow{\phi_-} \mcX_-$, where 
\[ \phi_{\pm} \colon \widetilde{\mcX} \simeq \Bl_{\mcZ_{\pm}^0} \mcX_{\pm} \to \mcX_{\pm} \]
is the blowing-up along a smooth subvariety $\mcZ_{\pm}^0 \subset \mcX_{\pm}$ for each $+$ and $-$.
\item[(2)] There is a linear equivalence $\phi_+^* K_{\mcX_+} \sim \phi_-^* K_{\mcX_-}$.
\end{enumerate}
If a simple K-equivalence is a flop (with respect to some boundaries), it is called a \textit{simple flop}.
\end{defi}

For a given $K$-equivalence, the associated subvarieties $\mcZ_{\pm}^0$ with their normal bundles in $\mcX_{\pm}$ are (scrolls of) simple Mukai pairs \cite[Thorem~0.2]{Kan22}.
The two blowing-ups $\phi_{\pm}$ has the same exceptional divisor $E \subset \widehat{\mcX}$, which is (a scroll of) the associated roof.

Conversely, one can construct a simple flop from a given roof as follows.
Let $(Z_+, \mcE_+)$ and $(Z_-, \mcE_-)$ be two simple Mukai pairs that give two distinct projective space bundle structures of the roof
$W = \P_{Z_+}(\mcE_+) \simeq \P_{Z_-}(\mcE_-)$ with projections $p_{\pm} \colon W \to Z_{\pm}$.
Then it is easy to see that
\begin{enumerate}
\item[(1)] $f_{\pm} \colon X_{\pm} \coloneqq \Tot_{Z_{\pm}}(\mcE_{\pm}^{\vee}) \to \Spec R$ is a flopping contraction for each $+$ and $-$, where $R = H^0(\mcO_{X_+}) \simeq H^0(\mcO_{X_-})$, 
\item[(2)] if $Z_{\pm}^0 \subset X_{\pm}$ denote the zero-sections, then
\[ \wX \coloneqq \Tot_{W}(\mcO(-1,-1)) \simeq \Bl_{Z_+^0} X_+ \simeq \Bl_{Z_-^0} X_- \]
as $R$-schemes, where $\mcO_W(a,b) \coloneqq p_+^*\mcO_{Z_+}(a) \otimes p_-^*\mcO_{Z_-}(b)$, and
\item[(3)] the birational map $X_+ \dashrightarrow X_-$ is a flop.
\end{enumerate}

In addition, the common exceptional divisor $E \subset \wX$ is the zero-section of $\wX \to W$,
and hence $E \simeq W$ and $\mcO_{\wX}(E) \simeq \mcO_{\wX}(-1,-1) \coloneqq (\wX \to W)^*\mcO_W(-1,-1)$.

\begin{defi} \label{def model flop}
\begin{enumerate}
\item[(1)] We call the flop $\Tot_{Z_+}(\mcE_+^{\vee}) \dashrightarrow \Tot_{Z_-}(\mcE_-^{\vee})$ the \textit{canonical local model} of the simple flop.
\item[(2)] Let $\mathfrak{m} \subset R$ be the maximal ideal corresponding to the unique singular point of $\Spec R$.
Let $\widehat{R}$ be the completion of the local ring $R_{\mathfrak{m}}$ along the maximal ideal,
and put $\widehat{X}_{\pm} \coloneqq \Tot_{Z_{\pm}}(\mcE_{\pm}^{\vee}) \otimes_R \widehat{R}$.
Then we call the flop $\widehat{X}_+ \dashrightarrow \widehat{X}_-$ the \textit{complete local model} of the simple flop.
\end{enumerate}
\end{defi}

Kanemitsu \cite{Kan22} named known simple Mukai pairs using Dynkin labels.
We follow this labeling and say that the canonical local model of a simple flop has the same Dynkin type.

\begin{rem}
The commutative algebra $R$ that appears in the flopping contractions of the canonical local model is the coordinate ring of the affine cone of the polarised manifold $(W, \mcO(1,1))$.
Since $Z_{\pm}$ are of Picard rank one, the manifolds $X_{\pm}$ has no more flops, and hence $R$ has exactly two minimal models (= $\Q$-factorial terminalizations), which are both crepant resolutions.
Similarly, $\widehat{R}$ has exactly two crepant resolutions, which are $\widehat{X}_{\pm}$.
\end{rem}

Let us fix the notation for the open locus that does not change under the flop.

\begin{defi} \label{common open sub}
Let $U$ be a scheme together with open immersions $j_{\pm} \colon U \to X_{\pm}$ such that $j_{\pm}(U) = X_{\pm} \setminus Z_{\pm}^0$, and call it the \textit{common open subset} of $X_{\pm}$.
\end{defi}

Given a simple Mukai pair $(Z_+,\mcE_+)$, let $(Z_-,\mcE_-)$ be the simple Mukai pair that give the other projective space bundle structure of $\P_{Z_+}(\mcE_+)$.
Then there are canonical isomorphisms $ H^0(\mcE_+) \simeq H^0(\mcO_W(1,1)) \simeq H^0(\mcE_-)$.
Pick a section $s_+ \in H^0(\mcE_+)$ and let $s_- \in H^0(\mcE_-)$ be the corresponding section.
If sections $s_{\pm}$ are general, then their zero-loci $V(s_{\pm}) \subset Z_{\pm}$ are (smooth) projective Calabi-Yau
varieties.
Those Calabi-Yau varieties need not to be birational in general (c.f.~\cite{IMOU20, KR22}), but are expected to satisfy several equivalences (e.g.~D-equivalence or L-equivalence).

\begin{rem}
As shown by Kanemitsu \cite{Kan22}, all roofs (and the corresponding simple Mukai pairs) whose associated Calabi-Yau manifolds $V(s_{\pm})$ are two-dimensional (i.e.~K3 surfaces) can be classified into exactly two types: $D_4$ and $G_2^{\dagger}$.
\end{rem}

\section{On type $D_4$ homogeneous varieties}

\subsection{Representation theory of type $D_4$} \label{rep of D}
Let us begin from recalling the root system and weight lattice of Dynkin type $D_4$.
The weight lattice $L_{D_4}$ of Dynkin type $D_4$ is a $\mathbb{Z}$-submodule of the real vector space $\R^4$
generated by vectors
\begin{center}
$(1,0,0,0)$, $(0,1,0,0)$, $(0,0,1,0)$, $(0,0,0,1)$ and $\left(\frac{1}{2},\frac{1}{2},\frac{1}{2},\frac{1}{2}\right)$.
\end{center}
The simple roots are given by
\begin{center}
$\alpha_1 = (1,-1,0,0)$, $\alpha_2 = (0,1,-1,0)$, $\alpha_3 = (0,0,1,-1)$, and $\alpha_4 = (0,0,1,1)$,
\end{center}
and the corresponding fundamental weights are
\begin{center}
$\omega_1 = (1,0,0,0)$, $\omega_2 = (1,1,0,0)$, $\omega_3 = \left(\frac{1}{2},\frac{1}{2},\frac{1}{2},\frac{1}{2}\right)$ and $\omega_4 = \left(\frac{1}{2},\frac{1}{2},\frac{1}{2},-\frac{1}{2}\right)$.
\end{center}
Thus, a weight $(\lambda_1, \lambda_2, \lambda_3, \lambda_4) \in L_{D_4}$ is dominant if and only if $\lambda_1 \geq \lambda_2 \geq \lambda_3 \geq \lvert \lambda_4 \rvert$.
Let $s_i \in W_{D_4}$ denote the element in the Weyl group $W_{D_4}$ corresponding to a simple root $\alpha_i$.
The weight lattice $L_{D_4}$ admits a natural action from $W_{D_4}$, but it also has another $W_{D_4}$-action called \textit{dotted action}, which can be described as follows.
\begin{align*}
s_1 \circ (\lambda_1, \lambda_2, \lambda_3, \lambda_4) &= (\lambda_2 - 1, \lambda_1 + 1, \lambda_3, \lambda_4) \\
s_2 \circ (\lambda_1, \lambda_2, \lambda_3, \lambda_4) &= (\lambda_1, \lambda_3 - 1, \lambda_2 + 1, \lambda_4) \\
s_3 \circ (\lambda_1, \lambda_2, \lambda_3, \lambda_4) &= (\lambda_1, \lambda_2, \lambda_4 -1, \lambda_3 + 1) \\
s_4 \circ (\lambda_1, \lambda_2, \lambda_3, \lambda_4) &= (\lambda_1, \lambda_2, -\lambda_4 -1, -\lambda_3 - 1)
\end{align*}
This action will used when applying the Borel-Bott-Weil theorem.

For each simple root $\alpha_i$, there is a corresponding parabolic subgroup $P_{\alpha_i} \subset \Spin(8)$
such that the associated homogeneous variety $D_4(i) \coloneqq \Spin(8)/P_{\alpha_i}$ has Picard rank one.
Note that the homogeneous varieties
\begin{center}
$D_4(1) = \OG(1,8)$, $D_4(3) = \OG_+(4,8)$, and $D_4(4) = \OG_-(4,8)$
\end{center}
are all isomorphic to the six-dimensional smooth quadric $\Q^6$.
Although the main interest of this article is the geometry of $\OG_{\pm}(4,8)$ and vector bundles over them,
only when applying the Borel-Bott-Weil theorem,
we identify them with $D_4(1) = \OG(1,8) \simeq \Q^6$,
since this identification most clearly reflects the symmetry of the weights.

It is known that, given an irreducible homogeneous vector bundle $\mcF$ on $\OG(1,8)$, 
there is a unique corresponding weight $\omega \in L_{D_4}$.
Using this correspondence, the bundle is denoted by $\mcF = \mcF_{\omega}$.
On the other hand, the Levi factor of $P_{\alpha_1}$ naturally contains $\GL(4)$, and hence an irreducible representation of $\GL(4)$ determines an irreducible homogeneous vector bundle over $\OG(1,8)$.
Recall that irreducible representations of $\GL(4)$ are parametrised by the set 
\[ L^+_{\GL(4)} = \{(a_1, a_2, a_3,a_4) \in \Z^4 \mid a_1 \geq a_2 \geq a_3 \geq a_4 \}, \]
and let $V_{(a_1,a_2,a_3,a_4)}$ denote the irreducible representation of $\GL(4)$ corresponding to $(a_1, a_2, a_3, a_4) \in L_{\GL(4)}^+$.
Then the irreducible homogeneous bundle over $\OG(1,8)$ that corresponds to an irreducible $\GL(4)$-representation 
$V_{(a_1,a_2,a_3,a_4)}$ is $\mcF_{\varphi(a_1,a_2,a_3,a_4)}$, where $\varphi$ is a linear map given by
\begin{equation} \label{interpret weight}
\varphi = 
\renewcommand{\arraystretch}{1.2}
\begin{pmatrix}
\frac{1}{2} & \frac{1}{2} & \frac{1}{2} & \frac{1}{2} \\
\frac{1}{2} & \frac{1}{2} & -\frac{1}{2} & -\frac{1}{2} \\
\frac{1}{2} & -\frac{1}{2} & \frac{1}{2} & -\frac{1}{2} \\
\frac{1}{2} & -\frac{1}{2} & -\frac{1}{2} & \frac{1}{2}
\end{pmatrix}
\colon L^+_{\GL(4)} \subset \Z^4 \to L_{D_4} \simeq \Z^4.
\end{equation}
This identification is useful for computing the decomposition of a tensor product of two irreducible homogeneous bundles over $\OG(1,8)$.

As established in \cite{Ottaviani88},
the quadric $6$-fold $\Q^6$ admits two spinor bundles.
If $\mcS$ denotes one of them, then the other is $\mcS^{\vee}(-1)$.
Then there exists an identification $\Q^6 \simeq \OG(1,8)$ such that
\begin{align*}
&\mcS(1) \simeq \mcF_{\left(\frac{1}{2},\frac{1}{2},\frac{1}{2},\frac{1}{2}\right)} = \mcF_{\varphi(1,0,0,0)}, \\
&\mcS^{\vee}(-1) \simeq \mcF_{\left(-\frac{1}{2},\frac{1}{2},\frac{1}{2},-\frac{1}{2}\right)} = \mcF_{\varphi(0,0,0,-1)}, \text{~and} \\
&\mcO(2) \simeq \mcF_{(2,0,0,0)} = \mcF_{\varphi(1,1,1,1)}.
\end{align*}
Under the same identification, 
\begin{align*}
\text{$\Sym^k \left(\mcS(1)\right) \otimes \mcO(j) \simeq \mcF_{\left(\frac{k+2j}{2},\frac{k}{2},\frac{k}{2},\frac{k}{2}\right)}$ and $\mcS^{\vee}(j) \simeq \mcF_{\left(\frac{1+2j}{2},\frac{1}{2},\frac{1}{2},-\frac{1}{2}\right)}$,}
\end{align*}
which will be used later.

\subsection{Borel-Bott-Weil computations}

The aim of this section is to compute cohomologies of homogeneous bundles over $\Q^6$ applying the Borel-Bott-Weil theorem.
For readers who are not familiar with the Borel-Bott-Weil theorem, see \cite[Section 4.3]{Weyman03}.

\begin{lem} \label{BBW}
The following hold.
\begin{enumerate}
\item[\rm (1)] $H^i(\Q^6, \Sym^k \mcS \otimes \mcO(2k+j)) = 0$ for all $i > 0$, $k \geq 0$ and $j \geq - 5$.
\item[\rm (2)] $H^i(\Q^6, \Sym^k \mcS \otimes  \mcS \otimes \mcO(2k+j)) = 0$ for all $i > 0$, $k \geq 0$ and $j \geq - 2$ except $(i,j,k) = (1,-2, 1)$.
\item[\rm (3)] $H^1(\Q^6, \mcS \otimes \mcS) \simeq \C$.
\item[\rm (4)] $H^i(\Q^6, \Sym^k \mcS \otimes  \mcS^{\vee} \otimes \mcO(2k+j-1)) = 0$ for all $i > 0$ and $j \geq -2$.
\item[\rm (5)] $H^i(\Q^6, \Sym^k \mcS \otimes \mcS \otimes \mcS^{\vee} \otimes \mcO(2k)) = 0$ for all $k \geq 0$.
\item[\rm (6)] $H^i(\Q^6, \Sym^k \mcS \otimes \mcS \otimes \mcS \otimes \mcO(2k+1)) = 0$ for all $k \geq 0$.
\item[\rm (7)] $H^i(\Q^6, \Sym^k \mcS \otimes \mcS^{\vee} \otimes \mcS^{\vee} \otimes \mcO(2k-1)) = 0$ for all $k \geq 0$.
\end{enumerate}
\end{lem}

\begin{proof}
Let us show (1).
The corresponding weight of the irreducible homogeneous bundle $\Sym^k \mcS \otimes \mcO(2k+j)$ is 
\begin{align} \label{weight 1}
\left(\frac{3k+2j}{2}, \frac{k}{2}, \frac{k}{2}, \frac{k}{2}\right) \in L_{D_4}.
\end{align}
For $k \geq 0$ and $j \geq - 5$, this weight is not  dominant if and only if $k + j <0$.
However, for all such $(k,j)$, one can see that the weight (\ref{weight 1}) is singular.
Thus (1) follows from the Borel-Bott-Weil theorem.
For example, when $(k,j) = (1,-5)$, the weight is $\left(\frac{-7}{2}, \frac{1}{2}, \frac{1}{2}, \frac{1}{2}\right)$.
For this, one has 
\[ (s_3s_2s_1) \circ \left(-\frac{7}{2}, \frac{1}{2}, \frac{1}{2}, \frac{1}{2}\right) = \left(-\frac{1}{2}, -\frac{1}{2}, -\frac{1}{2}, -\frac{1}{2}\right). \]
However, since
\[ s_4 \circ \left(-\frac{1}{2}, -\frac{1}{2}, -\frac{1}{2}, -\frac{1}{2}\right) = \left(-\frac{1}{2}, -\frac{1}{2}, -\frac{1}{2}, -\frac{1}{2}\right), \]
this weight is singular.

Next, let us show (2) and (3).
The strategy is as follows.
\begin{enumerate}
\item[(i)] Compute the decomposition of the tensor product of irreducible homogeneous bundles with respect to the weights in $L^+_{\GL(4)}$ using the Littlewood-Richardson rule for $\GL(4)$-representations, and
\item[(ii)] then interpret the weights into the ones in $L_{D_4}$ using (\ref{interpret weight}), and apply the Borel-Bott-Weil theorem.
\end{enumerate}
The bundle in (2) can be decomposed as
\begin{align*}
\Sym^k \mcS \otimes  \mcS \otimes \mcO(2k+j) \simeq
\begin{cases}
\mcF_{w_1} & \text{(if $k = 0$)} \\
\mcF_{w_1} \oplus \mcF_{w_2} & \text{(if $k \geq 1$)},
\end{cases}
\end{align*}
where
\[ 
w_1 = \left(\frac{3k+2j-1}{2}, \frac{k+1}{2}, \frac{k+1}{2}, \frac{k+1}{2}\right),~w_2 = \left(\frac{3k+2j-1}{2}, \frac{k+1}{2}, \frac{k-1}{2}, \frac{k-1}{2}\right) \in L_{D_4}.
\]
For $k \geq 0$ and $j \geq -2$, one can see that $w_1$ is dominant or singular as in the proof of (1).
Similarly, for $k \geq 1$ and $j \geq -2$, the weight $w_2$ is neither dominant nor singular if and only if $(k,j) = (1,-2)$.
In the exceptional case $(k,j) = (1,-2)$, we have $w_2 = \left(-1,1,0,0\right)$, and hence $s_1 \circ w_2 = (0,0,0,0)$.
Therefore the Borel-Bott-Weil theorem implies (3) and thus (2).

All the remaining cases can be shown in the same way, and we do not give further details, except the results for the decompositions of tensor products.
For simplicity, let $\left(\frac{k_1}{2}, \frac{k_2}{2}, \frac{k_3}{2}, \frac{k_4}{2}\right)$ denote the irreducible homogeneous bundle $\mcF_{\left(\frac{k_1}{2}, \frac{k_2}{2}, \frac{k_3}{2}, \frac{k_4}{2}\right)}$.
\begin{align*}
%%%
&\Sym^k \mcS \otimes  \mcS^{\vee} \otimes \mcO(2k+j-1) \\
{}\simeq{} &\begin{cases}
\left(\frac{2j-1}{2}, \frac{1}{2}, \frac{1}{2}, -\frac{1}{2}\right) & \text{(if $k = 0$)} \\
\left(\frac{3k+2j-1}{2}, \frac{k+1}{2}, \frac{k+1}{2}, \frac{k-1}{2}\right) \oplus \left(\frac{3k+2j-1}{2}, \frac{k-1}{2}, \frac{k-1}{2}, \frac{k-1}{2}\right) & \text{(if $k \geq 1$)}
\end{cases} \\
%%%
&\Sym^k \mcS \otimes \mcS \otimes \mcS^{\vee} \otimes \mcO(2k) \\
{}\simeq{} 
&\begin{cases}
\left(0,0,0,0\right) \oplus \left(0,1,1,0\right)
& \text{(if $k = 0$)} \\
\left(\frac{3}{2}, \frac{1}{2}, \frac{1}{2}, \frac{1}{2}\right)^{\oplus 2} \oplus \left(\frac{3}{2}, \frac{3}{2}, \frac{3}{2}, \frac{1}{2}\right) 
\oplus \left(\frac{3}{2}, \frac{3}{2}, \frac{1}{2}, -\frac{1}{2}\right)
& \text{(if $k = 1$)} \\
\left(\frac{3k}{2}, \frac{k}{2}, \frac{k}{2}, \frac{k}{2}\right)^{\oplus 2} \oplus \left(\frac{3k}{2}, \frac{k+2}{2}, \frac{k+2}{2}, \frac{k}{2}\right) 
\oplus \left(\frac{3k}{2}, \frac{k+2}{2}, \frac{k}{2}, \frac{k-2}{2}\right) \oplus \left(\frac{3k}{2}, \frac{k}{2}, \frac{k-2}{2}, \frac{k-2}{2}\right)
& \text{(if $k \geq 2$)}
\end{cases} \\
%%%
&\Sym^k \mcS \otimes \mcS \otimes \mcS \otimes \mcO(2k+1) \\
{}\simeq{} 
&\begin{cases}
\left(0,1,1,1\right) \oplus \left(0,1,0,0\right)
& \text{(if $k = 0$)} \\
\left(\frac{3}{2}, \frac{3}{2}, \frac{3}{2}, \frac{3}{2}\right) \oplus \left(\frac{3}{2}, \frac{3}{2}, \frac{1}{2}, \frac{1}{2}\right)^{\oplus 2}
\oplus \left(\frac{3}{2}, \frac{1}{2}, \frac{1}{2}, -\frac{1}{2}\right)
& \text{(if $k = 1$)} \\
\left(\frac{3k}{2}, \frac{k+2}{2}, \frac{k+2}{2}, \frac{k+2}{2}\right) 
\oplus \left(\frac{3k}{2}, \frac{k+2}{2}, \frac{k}{2}, \frac{k}{2}\right)^{\oplus 2} \oplus \left(\frac{3k}{2}, \frac{k}{2}, \frac{k}{2}, \frac{k-2}{2}\right)
\oplus \left(\frac{3k}{2}, \frac{k+2}{2}, \frac{k-2}{2}, \frac{k-2}{2}\right)
& \text{(if $k \geq 2$)}
\end{cases} \\
%%%
&\Sym^k \mcS \otimes \mcS^{\vee} \otimes \mcS^{\vee} \otimes \mcO(2k-1) \\
{}\simeq{} 
&\begin{cases}
\left(0,1,1,-1\right) \oplus \left(0,1,0,0\right)
& \text{(if $k = 0$)} \\
\left(\frac{3}{2}, \frac{3}{2}, \frac{3}{2}, -\frac{1}{2}\right) \oplus \left(\frac{3}{2}, \frac{3}{2}, \frac{1}{2}, \frac{1}{2}\right)
\oplus \left(\frac{3}{2}, \frac{1}{2}, \frac{1}{2}, -\frac{1}{2}\right)^{\oplus 2}
& \text{(if $k = 1$)} \\
\left(\frac{3k}{2}, \frac{k+2}{2}, \frac{k+2}{2}, \frac{k-2}{2}\right) \oplus \left(\frac{3k}{2}, \frac{k+2}{2}, \frac{k}{2}, \frac{k}{2}\right)
\oplus \left(\frac{3k}{2}, \frac{k}{2}, \frac{k}{2}, \frac{k-2}{2}\right)^{\oplus 2}
\oplus \left(\frac{3k}{2}, \frac{k-2}{2}, \frac{k-2}{2}, \frac{k-2}{2}\right)
& \text{(if $k \geq 2$)}
\end{cases}
\end{align*}
One can see that all the weights appeared above are dominant or singular.
\end{proof}

\subsection{The key bundle and the exchange diagram over $\OGr(3,8)$} \label{key bundle}
Let us recall the property of the bundle $\mcE$ over $\OG(3,8)$, which had been considered in the previous work \cite{Hara24} and will also play an important role in this article.
Consider the type $D_4$ roof $p_{\pm} \colon W = \OGr(3,8) \to Z_{\pm} = \OG_{\pm}(4,8)$.
Let $\mcV$ be the rank three universal subbundle over $W$.
Then there are two (dual of the) tautological sequences
\begin{center}
$0 \to \mcO_W(1,-1) \to p_+^*\mcS_+^{\vee} \to \mcV^{\vee} \to 0$ and
$0 \to \mcO_W(-1,1) \to p_-^*\mcS_-^{\vee} \to \mcV^{\vee} \to 0$.
\end{center}
Pulling-back the first exact sequence by the second gives a rank five vector bundle $\mcE$ on $Z$ and a commutative diagram
\[ \begin{tikzcd}
 & & 0 \arrow[d] & 0 \arrow[d] & \\
 & & \mcO_W(-1,1) \arrow[r, equal] \arrow[d] & \mcO_W(-1,1) \arrow[d] & \\
0 \arrow[r] & \mcO_W(1,-1) \arrow[r] \arrow[d, equal] & \mcE \arrow[r] \arrow[d] & p_-^*\mcS_-^{\vee} \arrow[r] \arrow[d] & 0 \\
0 \arrow[r] & \mcO_W(1,-1) \arrow[r] & p_+^*\mcS_+^{\vee} \arrow[r] \arrow[d] & \mcV^{\vee} \arrow[r] \arrow[d] & 0 \\
&&0&0&
\end{tikzcd} \]
of exact sequences.
The bundle $\mcE$ over $W$ will be referred as the \textit{key bundle}.

\subsection{Exceptional collections over $\OG_{\pm}(4,8)$}
Recall that $Z_{\pm} = \OG_{\pm}(4,8) \simeq \Q^6$.
The most well-known exceptional collection over $Z_{\pm}$ is the Kapranov collection (c.f.~\cite{Kap88})
\[ \Db(\coh Z_{\pm}) = \left\langle \mcS_{\pm}^{\vee}(-1), \mcS_{\pm}, \mcO, \mcO(1), \mcO(2), \mcO(3), \mcO(4), \mcO(5) \right\rangle. \]
However, we do not employ this exceptional collection in this article.
Instead, consider a mutation
\[
\begin{tikzpicture}
\draw (0,0) node[anchor=mid west] (a) {$\mcS_{\pm}^{\vee}(-1)$};
\draw ($(a.mid east)+(0.3,0)$) node[anchor= mid west] (b) {$\mcS_{\pm}$};
\draw ($(b.mid east)+(0.3,0)$) node[anchor=mid west] (c) {$\mcO$};
\draw ($(c.mid east)+(0.3,0)$) node[anchor=mid west] (d) {$\mcO(1)$};
\draw ($(d.mid east)+(0.3,0)$) node[anchor=mid west] (e) {$\mcO(2)$};
\draw ($(e.mid east)+(0.3,0)$) node[anchor=mid west] (f) {$\mcO(3)$};
\draw ($(f.mid east)+(0.3,0)$) node[anchor=mid west] (g) {$\mcO(4)$};
\draw ($(g.mid east)+(0.3,0)$) node[anchor=mid west] (h) {$\mcO(5)$};
%\draw[dashed, ->] (node cs:name=h, anchor=south) to[bend left=20] (node cs:name=e, anchor=south);
\draw[->] (node cs:name=g, anchor=south) to[bend left=10] node[yshift=-10pt]{${} \otimes \mcO(-6)$} (node cs:name=a, anchor=south west);
\draw (node cs:name=f, anchor=north west) -- (node cs:name=h, anchor=north east) -- (node cs:name=h, anchor=south east) -- (node cs:name=f, anchor=south west) -- (node cs:name=f, anchor=north west);

\end{tikzpicture}
\]
and then twisting the whole collection by $\mcO(1)$ gives the exceptional collection
\begin{align} \label{1st exc coll}
\Db(\coh Z_{\pm}) = \left\langle \mcO(-2), \mcO(-1), \mcO, \mcS_{\pm}^{\vee}, \mcS_{\pm}(1), \mcO(1), \mcO(2), \mcO(3) \right\rangle. 
\end{align}
This is the first collection that is used in this article.
We will need one more collection, which can be obtained by applying mutations to (\ref{1st exc coll}) as
\[
\begin{tikzpicture}
\draw (0,0) node[anchor=mid west] (a) {$\mcO(-2)$};
\draw ($(a.mid east)+(0.3,0)$) node[anchor= mid west] (b) {$\mcO(-1)$};
\draw ($(b.mid east)+(0.3,0)$) node[anchor=mid west] (c) {$\mcO$};
\draw ($(c.mid east)+(0.3,0)$) node[draw, anchor=mid west] (d) {$\mcS_{\pm}^{\vee}$};
\draw ($(d.mid east)+(0.3,0)$) node[draw, anchor=mid west] (e) {$\mcS_{\pm}(1)$};
\draw ($(e.mid east)+(0.3,0)$) node[anchor=mid west] (f) {$\mcO(1)$};
\draw ($(f.mid east)+(0.3,0)$) node[anchor=mid west] (g) {$\mcO(2)$};
\draw ($(g.mid east)+(0.3,0)$) node[anchor=mid west] (h) {$\mcO(3).$};
\draw[->] (node cs:name=e, anchor=south) to[bend right=20] node[yshift=-10pt]{$\RR_{\mcO(1)}$} (node cs:name=f, anchor=south east);
\draw[->] (node cs:name=d, anchor=south) to[bend left=40] node[yshift=-10pt]{$\LL_{\mcO}$} (node cs:name=c, anchor=south west);
\end{tikzpicture}
\]
These mutations can be computed explicitly using the exact sequences in \cite[Theorem~2.8]{Ottaviani88}, and after twisting the whole collection by $\mcO(-1)$, we obtain a new exceptional collection
\begin{align} \label{2nd exc coll}
\Db(\coh Z_{\pm}) = \left\langle \mcO(-3), \mcO(-2), \mcS_{\pm}(-1), \mcO(-1), \mcO, \mcS_{\pm}^{\vee}, \mcO(1), \mcO(2) \right\rangle. 
\end{align}
These two collections (\ref{1st exc coll}) and (\ref{2nd exc coll}) play important roles to prove Theorem~\ref{main thm}.

\section{Construction of tilting bundles} \label{sect: tilting}

Consider the two connected components $Z_{\pm} = \OG_{\pm}(4,8)$ of the orthogonal grassmannian $\OG(4,8)$,
and let $\mcS_{\pm}$ be the universal subbundles.
Put
\[ X_{\pm} = \Tot_{Z_{\pm}}(\mcS_{\pm}(2)) = \Spec_{Z_{\pm}} \Sym^{\bullet} \mcS_{\pm}(2). \]
Let $\mcO_{X_{\pm}}(a)$ (resp.~$\mcS_{X_{\pm}}$) denote the pull-back of $\mcO_{Z_{\pm}}(a)$ (resp.~$\mcS_{\pm}$) on $Z_{\pm}$ by the projection $X_{\pm} \to Z_{\pm}$.

\subsection{The first tilting bundles}

\begin{lem} \label{pretilting}
The following hold.
\begin{enumerate}
\item[\rm (1)] $H^i(X_{\pm}, \mcO_{X_{\pm}}(j)) = 0$ for all $i > 0$ and $j \geq -5$.
\item[\rm (2)] $H^i(X_{\pm}, \mcS_{X_{\pm}}(j)) = 0$ for all $i>0$ and $j \geq -2$ except when $(i,j) = (1,-2)$.
\item[\rm (3)] $H^1(X_{\pm}, \mcS_{X_{\pm}}(-2)) = \C$.
\item[\rm (4)] $H^i(X_{\pm}, \mcS_{X_{\pm}}^{\vee}(j)) = 0$ for all $i>0$ and $j \geq -3$.
\item[\rm (5)] $\Ext_{X_{\pm}}^i(\mcS_{X_{\pm}}, \mcS_{X_{\pm}}) = 0$ for all $i>0$.
\item[\rm (6)] $\Ext_{X_{\pm}}^i(\mcS^{\vee}_{X_{\pm}}, \mcS_{X_{\pm}}(1)) = 0$ for all $i>0$.
\item[\rm (7)] $\Ext_{X_{\pm}}^i(\mcS_{X_{\pm}}(1), \mcS^{\vee}_{X_{\pm}}) = 0$ for all $i>0$.
\end{enumerate}
\end{lem}

\begin{proof}
Let $\mcK$ be a vector bundle over $Z_{\pm}$, and $\mcK_{X_{\pm}}$ the pull-back of $\mcK$ by the natural projection $X_{\pm} \to Z_{\pm}$.
Then the projection formula gives
\[ H^i(X_{\pm}, \mcK_{X_{\pm}}) \simeq \bigoplus_{k \geq 0}H^i\left(Z_{\pm}, \Sym^k\left(\mcS_{\pm}(2)\right) \otimes \mcK \right). \]
Fix identifications $\psi_{\pm} \colon Z_{\pm} \simeq \OG(1,8)$ for each $+$ and $-$ 
such that $\psi_{\pm}^*\mcS \simeq \mcS_{\pm}$.
Then the results follow from Lemma~\ref{BBW}.
\end{proof}

By the above lemma, it follows that $\Ext_{X_{\pm}}^1(\mcS_{X_{\pm}}^{\vee}, \mcO_{X_{\pm}}(-2)) \simeq H^1(X_{\pm}, \mcS_{X_{\pm}}(-2)) = \C$.

\begin{defi}
For each $+$ and $-$, let $\mcP_{\pm}$ be the vector bundle over $X_{\pm}$ that lies in the unique non-trivial extension
\begin{align} \label{def of P}
0 \to \mcO_{X_{\pm}}(-2) \to \mcP_{\pm} \to \mcS_{X_{\pm}}^{\vee} \to 0. 
\end{align}
\end{defi}

Dualising the sequence (\ref{def of P}) and then twisting by $\mcO_{X_{\pm}}(1)$ gives
\begin{align} \label{def of Pv}
0 \to \mcS_{X_{\pm}}(1) \to \mcP_{\pm}^{\vee}(1) \to \mcO_{X_{\pm}}(3) \to 0,
\end{align}
and this sequence gives a non-trivial element of $\Ext^1(\mcO_{X_{\pm}}(3), \mcS_{X_{\pm}}(1)) = \C$.

\begin{lem} \label{semiuniv ext}
Let  $\mcF_1$ and $\mcF_2$ be vector bundles over a $\C$-scheme $S$ such that $\Ext^1(\mcF_2, \mcF_1) = \C$, and $0 \to \mcF_1 \to \mcF_3 \to \mcF_2\to 0$ the associated unique non-trivial extension. 
Assume that the following two conditions hold.
\begin{enumerate}
\item[\rm (a)] $\mcF_1$ and $\mcF_2$ are pretilting.
\item[\rm (b)] $\Ext^i(\mcF_1, \mcF_2) = 0$ for all $i >0$ and $\Ext^i(\mcF_2, \mcF_1) = 0$ for all $i > 1$.
%\item[\rm (c)] the induced map $\Hom(\mcF_2, \mcF_2)^{\oplus r} \to \Ext^1(\mcF_2, \mcF_1)$ is surjective.
\end{enumerate}
Then the following hold.
\begin{enumerate}
\item[\rm (1)] The bundle $\mcF_3$ satisfies the vanishings
\begin{enumerate}
\item[$\bullet$] $\Ext^i(\mcF_3, \mcF_1) = 0$ for all $i > 0$, and
\item[$\bullet$] $\Ext^i(\mcF_2, \mcF_3) = 0$ for all $i > 0$.
\end{enumerate}
\item[\rm (2)] The bundles $\mcF_1 \oplus \mcF_3$ and $\mcF_2 \oplus \mcF_3$ are pretilting. In particular, $\mcF_3$ itself is pretilting.
\end{enumerate}
\end{lem}

\begin{proof}
By assumptions (a) and (b), it follows that $\Ext^i(\mcF_3, \mcF_2) = 0$ for all $i > 0$.
Similarly, applying the functor $\Ext^i(-,\mcF_1)$ to the given exact sequence yields a long exact sequence
\begin{align*}
\cdots &\to \Hom(\mcF_1, \mcF_1) \xrightarrow{\delta} \Ext^1(\mcF_2, \mcF_1) \to \Ext^1(\mcF_3, \mcF_1) \to \Ext^1(\mcF_1, \mcF_1) = 0 \to \cdots \\
&\to \Ext^i(\mcF_2, \mcF_1) = 0 \to \Ext^i(\mcF_3, \mcF_1) \to \Ext^i(\mcF_1, \mcF_1) = 0 \to \cdots.
\end{align*}
Since $\delta$ is surjective by the assumption, the vanishing $\Ext^i(\mcF_3, \mcF_1) = 0$ holds for all $i > 0$.
Finally, applying the functor $\Ext^i(-,\mcF_3)$ to the given sequence shows that 
$\Ext^i(\mcF_3,\mcF_3) = 0$, i.e.~the bundle $\mcF_3$ is pretilting.

Dually, it easily follows that $\Ext^i(\mcF_1, \mcF_3) = 0$ for all $i > 0$.
Applying the functor $\Ext^i(\mcF_2, -)$ to the given sequence gives a long exact sequence
\begin{align*}
\cdots &\to \Hom(\mcF_2, \mcF_2) \xrightarrow{\delta'} \Ext^1(\mcF_2, \mcF_1) \to \Ext^1(\mcF_2, \mcF_3) \to \Ext^1(\mcF_2, \mcF_2) = 0 \to \cdots \\
&\to \Ext^i(\mcF_2, \mcF_1) = 0 \to \Ext^i(\mcF_2, \mcF_3) \to \Ext^i(\mcF_2, \mcF_2) = 0 \to \cdots.
\end{align*}
As before, $\delta'$ is surjective by the assumption, the vanishing $\Ext^i(\mcF_2, \mcF_3) = 0$ holds for all $i > 0$.
\end{proof}

\begin{thm} \label{thm 1st tilting}
The bundles
\begin{align*}
\mcT^{\sharp}_{+} &\coloneqq  \mcO_{X_{+}}(-2) \oplus \mcO_{X_{+}}(-1)  \oplus \mcO_{X_{+}}  \oplus \mcP_+ \oplus \mcP_+^{\vee}(1) \oplus \mcO_{X_{+}}(1) \oplus \mcO_{X_{+}}(2) \oplus \mcO_{X_{+}}(3), ~ \text{and} \\
\mcT^{\flat}_{-} &\coloneqq  \mcO_{X_{-}}(-3) \oplus \mcO_{X_{-}}(-2)  \oplus \mcO_{X_{-}}(-1)  \oplus \mcP_-(-1) \oplus \mcP_-^{\vee} \oplus \mcO_{X_{-}} \oplus \mcO_{X_{-}}(1) \oplus \mcO_{X_{-}}(2)
\end{align*}
are tilting bundles over $X_+$ and $X_-$, respectively.
\end{thm}

\begin{proof}
The direct sum of the collection (\ref{1st exc coll})
\[ \mcO(-2) \oplus \mcO(-1) \oplus \mcO \oplus \mcS_+^{\vee} \oplus \mcS_+(1) \oplus \mcO(1) \oplus \mcO(2) \oplus \mcO(3) \]
is a generator of $\D(\Qcoh Z_+)$.
Thus, by \cite[Lemma~3.1]{Hara21}, the bundle
\[ \mcU \coloneqq \mcO_{X_+}(-2) \oplus \mcO_{X_+}(-1) \oplus \mcO_{X_+} \oplus \mcS_{X_+}^{\vee} \oplus \mcS_{X_+}(1) \oplus \mcO_{X_+}(1) \oplus \mcO_{X_+}(2) \oplus \mcO_{X_+}(3) \]
is a generator of $\D(\Qcoh X_+)$.
By the sequences (\ref{def of P}) and (\ref{def of Pv}), the triangulated subcategory $\langle \mcT_+^{\sharp} \rangle \subset \Db(\coh X_+)$ contains $\mcU$.
Therefore $\mcT_+^{\sharp}$ is also a generator.

Thus, to prove that $\mcT_+^{\sharp}$ is tilting, it is enough to show that $\mcT_+^{\sharp}$ is pretilting.
Note that all the summands of $\mcT_+^{\sharp}$ are pretilting by Lemma~\ref{semiuniv ext}.
Thus, together with Lemma~\ref{pretilting}~(1), it remains to check the following vanishings.
\begin{enumerate}
\item[(i)] $H^i(X_+, \mcP_+(a)) = 0$ for all $i > 0$ and $-3 \leq a \leq 2$,
\item[(ii)] $H^i(X_+, \mcP_+^{\vee}(b)) = 0$ for all $i > 0$ and $-2 \leq b \leq 3$,
\item[(iii)] $\Ext^i_{X_+}(\mcP_+, \mcP_+^{\vee}(1)) = 0$ for all $i > 0$, and
\item[(iv)] $\Ext^i_{X_+}(\mcP_+^{\vee}(1), \mcP_+) = 0$ for all $i > 0$.
\end{enumerate}
First, (i) follows from Lemma~\ref{pretilting}~(1) and (4) together with the sequence (\ref{def of P}).

Let us show (ii).
As in (i), Lemma~\ref{pretilting}~(1) and (2) together with the sequence (\ref{def of Pv}) show that 
$H^i(X_+, \mcP_+^{\vee}(b)) = 0$ for all $i > 0$ and $-1 \leq b \leq 3$.
The last case when $b = -2$ follows from Lemma~\ref{semiuniv ext}~(1).

It remains to show (iii) and (iv).
Applying the functor $\Ext^*_{X_+}(\mcS_{X_+}(1),-)$ to (\ref{def of P}) and then using Lemma~\ref{pretilting}~(4) and (7) show that
$\Ext^i_{X_+}(\mcS_{X_+}(1),\mcP_+) = 0$ for all $i > 0$.
Thus, applying the functor $\Ext^*_{X_+}(-,\mcP_+)$ to (\ref{def of Pv}) together with (i) proves the vanishing $\Ext^i_{X_+}(\mcP_+, \mcP_+^{\vee}(1)) = 0$ for all $i > 0$, which is (iii).
(iv) can be shown similarly.
We have now completed the proof that $\mcT_+^{\sharp}$ is a tilting bundle over $X_+$.

Let us consider a non-canonical identification $\varphi \colon Z_- \xrightarrow{\sim} Z_+$ such that $\varphi^*\mcS_+ \simeq \mcS_-$.
This induces an isomorphism $\widetilde{\varphi} \colon X_- \xrightarrow{\sim} X_+$ of $\C$-schemes,
which satisfies $\widetilde{\varphi}^*\mcT_+^{\sharp} \simeq \mcT_-^{\flat} \otimes \mcO_{X_-}(1)$.
Thus $\mcT_-^{\flat}$ is also a tilting bundle over $X_-$.
\end{proof}

Although we have constructed tilting bundles over both $X_+$ and $X_-$,
they are not enough to prove Theorem~\ref{main thm},
and it is required to construct other tilting bundles.
The next two sections will be devoted to the construction of the second tilting bundles.

\subsection{Comparison of vector bundles}
In the previous section, the bundle $\mcP_{\pm}$ had been constructed to give summands of the tilting bundles.
The present section studies the strict transforms of line bundles and the bundle $\mcP_{\pm}$ under the flop.
Recall from Section~\ref{sect: flop} that the diagram of the flop is 
\[ \begin{tikzcd}
 & E = W \arrow[d, hook] \arrow[dl, "\pi_+"'] \arrow[dr, "\pi_-"]&  \\
Z_+^0 \arrow[d, hook] & \wX = \Tot_W(\mcO_W(-1,-1)) \arrow[dl, "\phi_+"', "\Bl_{Z_+^0}"] \arrow[dr, "\phi_-", "\Bl_{Z_-^0}"'] & Z_-^0 \arrow[d, hook] \\
X_+ = \Tot_{Z_+}(\mcS_+^{\vee}(-2)) \arrow[d] \arrow[rr, dashed, "\text{flop}"'] \arrow[dr, "f_+"'] & & X_- = \Tot_{Z_-}(\mcS_-^{\vee}(-2)) \arrow[d] \arrow[dl, "f_-"] \\
 Z_+ = \OG_+(4,8) & \Spec R & Z_- = \OG_-(4,8).
\end{tikzcd} \]
The following is elementary.

\begin{lem} \label{push of exc div}
For each $+$ and $-$, the following hold.
\begin{enumerate}
\item[\rm (1)] $R(\phi_{\pm})_* \mcO_{\wX}(kE) \simeq \mcO_{X_{\pm}}$ for all $0 \leq k \leq 3$.
\item[\rm (2)] There exists an exact triangle
\[ \mcO_{X_{\pm}} \to R(\phi_{\pm})_* \mcO_{\wX}(4E) \to \mcO_{Z_{\pm}^0}(-6)[-3] \to \mcO_{X_{\pm}}[1]. \]
\end{enumerate}
\end{lem}

\begin{proof}
Since the projections $\pi_{\pm} \colon E \to Z_{\pm}^0$ are both $\P^3$-bundles, $R(\pi_{\pm})_*\mcO_E(kE) = 0$ for all
$1 \leq k \leq 3$ and
\[ R(\pi_{\pm})_*\mcO_E(4E) \simeq R^3(\pi_{\pm})_*\mcO_E(4E)[-3] \simeq \det \left(\mcS_{\pm}^{\vee}(-2)\right)[-3] \simeq \mcO_{Z_{\pm}^0}(-6)[-3]. \]
Thus the statement follows from the exact sequence $0 \to \mcO_{\wX}((k-1)E) \to \mcO_{\wX}(kE) \to \mcO_E(kE) \to 0$ and the induction on $k$.
\end{proof}

Let $\mcE$ be the key bundle over $W$,
and $\mcE_{\wX}$ the pull-back of the key bundle $\mcE$ to $\wX$ 
by the projection $\wX = \Tot_W(\mcO_W(-1,-1)) \to W$.
The pull-back of the exact sequences over $W$ (in Section~\ref{key bundle}) to $\wX$ by the projection gives two exact sequences
\begin{align*}
&0 \to \mcO_{\wX}(-1,1) \to \mcE_{\wX} \to \phi_+^*\mcS_{X_+}^{\vee} \to 0, \\
&0 \to \mcO_{\wX}(1,-1) \to \mcE_{\wX} \to \phi_-^*\mcS_{X_-}^{\vee} \to 0.
\end{align*}

\begin{prop} \label{prop exchange}
There are the following isomorphisms.
\begin{enumerate}
\item[\rm (1)] $R(\phi_{\pm})_*\left(\mcE_{\wX}(E)\right) \simeq \mcP_{\pm}$  for each $+$ and $-$.
\item[\rm (2)] $R(\phi_+)_*\left(\mcE_{\wX}^{\vee}(-1,0)\right) \simeq \mcP_+^{\vee}(-1)$ and $R(\phi_-)_*\left(\mcE_{\wX}^{\vee}(-1,0)\right) \simeq \mcP_-^{\vee}(1)$.
\item[\rm (3)] $R(\phi_+)_*\left(\mcE_{\wX}^{\vee}(0,-1)\right) \simeq \mcP_+^{\vee}(1)$ and $R(\phi_-)_*\left(\mcE_{\wX}^{\vee}(0,-1)\right) \simeq \mcP_-^{\vee}(-1)$.
\end{enumerate}

\end{prop}

\begin{proof}
Let us show (2), and the proofs of the others are similar (c.f.~\cite[Prop.~3.6]{Hara24}).
Consider the exact sequence
\[ 0 \to \phi_+^*\mcS_{X_+}(-1,0) \to \mcE_{\wX}^{\vee}(-1,0) \to \mcO_{\wX}(0,-1) \to 0. \]
Since there are isomorphisms
\begin{align*}
&R(\phi_+)_*\left(\phi_+^*\mcS_{X_+}(-1,0)\right) \simeq \mcS_{X_+}(-1), \text{~ and} \\
&R(\phi_+)_*\left(\mcO_{\wX}(0,-1)\right) \simeq R(\phi_+)_*\left(\mcO_{\wX}(1,0) \otimes \mcO_{\wX}(E) \right) \simeq  
\mcO_{X_+}(1) \otimes R(\phi_+)_*\left( \mcO_{\wX}(E) \right) \simeq \mcO_{X_+}(1),
\end{align*}
it holds that $R(\phi_+)_*\left(\mcE_{\wX}^{\vee}(-1,0)\right) \simeq (\phi_+)_*\left(\mcE_{\wX}^{\vee}(-1,0)\right)$ and this sheaf lies in the exact sequence
\begin{align} \label{nonsplit 1}
0 \to \mcS_{X_+}(-1) \to (\phi_+)_*\left(\mcE_{\wX}^{\vee}(-1,0)\right) \to \mcO_{X_+}(1) \to 0. 
\end{align}
Thus together with Lemma~\ref{pretilting}~(3), it is enough to prove that the exact sequence (\ref{nonsplit 1}) is non-trivial.
In order to see this, consider the map
\[ \varepsilon_+ \colon \Ext_{\wX}^1(\mcO_{\wX}(0,-1), \phi_+^*\mcS_{X_+}(-1,0)) \to \Ext^1(\mcO_{X_+}(1), \mcS_{X_+}(-1)) \simeq \C \]
associated to the derived functor $R(\phi_+)_*$.
Then this map $\varepsilon_+$ is an isomorphism.
Indeed, this map is obtained by applying the functor $\Ext_{\wX}^{\ast}(-, \phi_+^*\mcS_{X_+}(-1,0))$ to the counit map 
\[ \phi_+^*(\phi_+)_*\mcO_{\wX}(0,-1) \simeq \mcO_{\wX}(1,0) \to \mcO_{\wX}(0,-1) \simeq \mcO_{\wX}(1,0) \otimes \mcO_{\wX}(E), \]
which is injective and has the cokernel $\mcO_{E}(1,0) \otimes_E \mcO_{E}(E) \simeq \mcO_E(0,-1)$.
Now it holds that
\begin{align*}
&\RHom_{\wX}\left(\mcO_{E}(1,0) \otimes_E \mcO_{E}(E), \phi_+^*\mcS_{X_+}(-1,0)\right) &\\
{}\simeq{} &\RHom_{E}\left(\mcO_{E}(1,0) \otimes_E \mcO_{E}(E), \phi_+^*\mcS_{X_+}(-1,0)|_E \otimes \mcO_{E}(E)[-1]\right) &\text{(adjunction)} \\
{}\simeq{} &\RG\left(E, \pi_+^*\mcS_{X_+}(-2)|_{Z_+^0} \right)[-1] & \\
{}\simeq{} &\RG\left(Z_+, \mcS_{+}(-2) \right)[-1] & \\
{}={} &0.
\end{align*}
Therefore $\varepsilon_+$ is an isomorphism, which implies that the sequence (\ref{nonsplit 1}) is non-trivial.

The other side is slightly more involved.
First, applying the functor $R(\phi_-)_*$ to the exact sequence
\[ \xi \colon 0 \to \phi_-^*\mcS_{X_-}(-1,0) \to \mcE_{\wX}^{\vee}(-1,0) \to \mcO_{\wX}(-2,1) \to 0 \]
gives an exact sequence
\begin{align} \label{nonsplit 2}
\eta \coloneqq (\phi_-)_*\xi \colon 0 \to \mcS_{X_-}(1) \to (\phi_-)_*\left(\mcE_{\wX}^{\vee}(-1,0)\right) \to \mcO_{X_-}(3) \to 0,
\end{align}
since 
\begin{align*}
&R(\phi_-)_*\phi_-^*\mcS_-(-1,0) \simeq R(\phi_-)_*\left(\phi_-^*\mcS_-(0,1) \otimes \mcO_{\wX}(E)\right) \simeq \mcS_-(1), ~\text{and} \\
&R(\phi_-)_*\mcO_{\wX}(-2,1) \simeq R(\phi_-)_*\left(\mcO_{\wX}(0,3) \otimes \mcO_{\wX}(2E)\right) \simeq \mcO_{X_-}(3)
\end{align*}
by Lemma~\ref{push of exc div}.
In order to see that the sequence (\ref{nonsplit 2}) is non-trivial,
consider
\[ \xi \otimes \mcO_{\wX}(E) \colon 0 \to \phi_-^*\mcS_{X_-}(0,1) \to \mcE_{\wX}^{\vee}(0,1) \to \mcO_{\wX}(-1,2) \to 0. \]
As before, applying the functor $R(\phi_-)_*$ to this sequence gives an exact sequence
\begin{align} \label{nonsplit 3}
\eta' \coloneqq (\phi_-)_*(\xi \otimes \mcO(E)) \colon 0 \to \mcS_{X_-}(1) \to (\phi_-)_*\left(\mcE_{\wX}^{\vee}(0,1)\right) \to \mcO_{X_-}(3) \to 0.
\end{align}
Now the correspondence $\xi \otimes \mcO(E) \mapsto \eta'$ is given by
\[ \varepsilon_- \colon \Ext_{\wX}^1(\mcO_{\wX}(-1,2), \phi_-^*\mcS_{X_-}(0,1)) \to \Ext_{X_-}^1(\mcO_{X_-}(3), \mcS_{X_-}(1)) \simeq \C, \]
and this homomorphism is obtained by applying $\Ext^1_{\wX}(-, \phi_-^*\mcS_{X_-}(0,1))$ to
the injective counit morphism
\[ \mcO_{X_-}(3) \simeq \phi_-^*(\phi_-)_*\mcO_{\wX}(-1,2) \to \mcO_{\wX}(-1,2) \simeq \mcO_{\wX}(0,3) \otimes \mcO_{\wX}(E), \]
whose cokernel is $\mcO_E(-1,2)$.
Since
\begin{align*}
&\RHom_{\wX}(\mcO_E(-1,2), \phi_-^*\mcS_{X_-}(0,1)) \\
{}\simeq{} & \RHom_{E}(\mcO_E(-1,2), \phi_-^*\mcS_{X_-}(0,1)|_E \otimes_E \mcO_{E}(E)[-1]) &\text{(adjunction)} \\
{}\simeq{} & \RG(E, \pi_-^*\mcS_{X_-}(0,-2)|_{Z_-^0})[-1] & \\
{}\simeq{} & \RG(Z_-, \mcS_{-}(-2))[-1] & \\
{}={} & 0,
\end{align*}
the map $\varepsilon_-$ is an isomorphism, and hence $\eta'$ is a non-trivial exact sequence.

To show that $\eta$ is non-trivial, let $U$ be the common open subset (see Definition~\ref{common open sub})
with open embeddings $j_- \colon U \to X_-$ and $\widetilde{j} \colon U \to \wX$.
Since the restriction $j_-^* \colon \refl(X_-) \to \refl(U)$ is an equivalence of categories,
it follows that
\begin{enumerate}
\item[$\bullet$] the sequence $\eta$ is non-trivial if and only if the restricted sequence $j_-^*\eta$ is, and
\item[$\bullet$] the sequence $j_-^*\eta'$ is non-trivial.
\end{enumerate}
Now, since $\phi_-$ restricts to the isomorphism $\widetilde{j}(U) \to j_-(U)$,
it holds that
\[ j_-^*\eta = \widetilde{j}^*\xi = \widetilde{j}^*(\xi \otimes \mcO_{\wX}(E)) = j_-^*\eta'. \]
Therefore, the exact sequence $\eta$ is also non-trivial,
and hence the uniqueness of a non-trivial sequence implies that
$R(\phi_-)_*\left(\mcE_{\wX}^{\vee}(-1,0)\right) \simeq \mcP_-^{\vee}(1)$.
\end{proof}

\begin{lem} \label{compare bundles}
Let $U$ be the common open subset with open immersions $j_{\pm} \colon U \to X_{\pm}$.
Then the following hold.
\begin{enumerate}
\item[\rm (1)] $j_+^*\mcO_{X_+}(a) \simeq j_-^*\mcO_{X_-}(-a)$ for all $a \in \Z$.
\item[\rm (2)] $j_+^*\mcP_+(a) \simeq j_-^*\mcP_-(-a)$ for all $a \in \Z$.
\item[\rm (3)] $j_+^*\mcP^{\vee}_+(a) \simeq j_-^*\mcP_-^{\vee}(-a)$ for all $a \in \Z$.
\end{enumerate}
\end{lem}

\begin{proof}
Note that $j_{\pm}^* \colon \refl(X_{\pm}) \to \refl(U)$ is an equivalence of categories.
Now (1) holds since $X_+ \dashrightarrow X_-$ is a flop.
(2) and (3) follow from Proposition~\ref{prop exchange}, since there is an open immersion $\widetilde{j} \colon U \to \wX$
such that $\phi_{\pm} \circ \widetilde{j} = j_{\pm}$.
\end{proof}

\subsection{The second tilting bundles}

\begin{thm} \label{thm 2nd tilting}
The bundles
\begin{align*}
\mcT^{\flat}_{+} &\coloneqq  \mcO_{X_{+}}(-2) \oplus  \mcO_{X_{+}}(-1)  \oplus \mcP_+^{\vee} \oplus \mcO_{X_{+}}  \oplus   \mcO_{X_{+}}(1) \oplus \mcP_+(1) \oplus \mcO_{X_{+}}(2) \oplus \mcO_{X_{+}}(3), ~ \text{and} \\
\mcT^{\sharp}_{-} &\coloneqq  \mcO_{X_{-}}(-3) \oplus \mcO_{X_{-}}(-2) \oplus \mcP_-^{\vee}(-1)  \oplus \mcO_{X_{-}}(-1)  \oplus   \mcO_{X_{-}} \oplus \mcP_- \oplus \mcO_{X_{-}}(1) \oplus \mcO_{X_{-}}(2)
\end{align*}
are tilting bundles over $X_+$ and $X_-$, respectively.
\end{thm}

As a preparation, let us show the following lemma.

\begin{lem} \label{BBW2}
The following hold.
\begin{enumerate}
\item[\rm (1)] $H^i(\Q^6, \Sym^k \mcS \otimes \mcS \otimes \mcS \otimes \mcO(2k-1)) = 0$ for all $i > 1$ and $k \geq 0$.
\item[\rm (2)] $H^i(\Q^6, \Sym^k \mcS \otimes \mcS^{\vee} \otimes \mcS^{\vee} \otimes \mcO(2k+1)) = 0$ for all $i > 0$ and $k \geq 0$.
\end{enumerate}
In addition, the above implies the following.
\begin{enumerate}
\item[\rm (I)] $\Ext_{X_{-}}^i(\mcS^{\vee}_{X_{-}}, \mcS_{X_{-}}(-1)) = 0$ for all $i>1$.
\item[\rm (II)] $\Ext_{X_{-}}^i(\mcS_{X_{-}}(-1), \mcS^{\vee}_{X_{-}}) = 0$ for all $i>0$.
\end{enumerate}
\end{lem}

\begin{proof}
As in Section~\ref{rep of D}, fix an identification $\Q^6 \simeq \OG(1,8)$.
Then as in the proof of Lemma~\ref{BBW},
\begin{align*}
%%%
&\Sym^k \mcS \otimes \mcS \otimes \mcS \otimes \mcO(2k-1) \\
{}\simeq{} 
&\begin{cases}
\left(-2,1,1,1\right) \oplus \left(-2,1,0,0\right)
& \text{(if $k = 0$)} \\
\left(-\frac{1}{2}, \frac{3}{2}, \frac{3}{2}, \frac{3}{2}\right) \oplus \left(-\frac{1}{2}, \frac{3}{2}, \frac{1}{2}, \frac{1}{2}\right)^{\oplus 2}
\oplus \left(-\frac{1}{2}, \frac{1}{2}, \frac{1}{2}, -\frac{1}{2}\right)
& \text{(if $k = 1$)} \\
\left(\frac{3k-4}{2}, \frac{k+2}{2}, \frac{k+2}{2}, \frac{k+2}{2}\right) 
\oplus \left(\frac{3k-4}{2}, \frac{k+2}{2}, \frac{k}{2}, \frac{k}{2}\right)^{\oplus 2} \oplus \left(\frac{3k-4}{2}, \frac{k}{2}, \frac{k}{2}, \frac{k-2}{2}\right)
\oplus \left(\frac{3k-4}{2}, \frac{k+2}{2}, \frac{k-2}{2}, \frac{k-2}{2}\right)
& \text{(if $k \geq 2$).}
\end{cases} 
\end{align*}
A weight that is neither dominant nor singular appears only when $k = 1$, and that weight is $\left(-\frac{1}{2}, \frac{3}{2}, \frac{1}{2}, \frac{1}{2}\right)$.
Since 
\[ s_1 \circ \left(-\frac{1}{2}, \frac{3}{2}, \frac{1}{2}, \frac{1}{2}\right) = \left(\frac{1}{2}, \frac{1}{2}, \frac{1}{2}, \frac{1}{2}\right), \]
$H^i(\mcS \otimes \mcS \otimes \mcS \otimes \mcO(-1)) = 0$ for all $i \neq 1$.
This proves (1).
For (2), similarly one has
\begin{align*}
&\Sym^k \mcS \otimes \mcS^{\vee} \otimes \mcS^{\vee} \otimes \mcO(2k+1) \\
{}\simeq{} 
&\begin{cases}
\left(2,1,1,-1\right) \oplus \left(2,1,0,0\right)
& \text{(if $k = 0$)} \\
\left(\frac{7}{2}, \frac{3}{2}, \frac{3}{2}, -\frac{1}{2}\right) \oplus \left(\frac{7}{2}, \frac{3}{2}, \frac{1}{2}, \frac{1}{2}\right)
\oplus \left(\frac{7}{2}, \frac{1}{2}, \frac{1}{2}, -\frac{1}{2}\right)^{\oplus 2}
& \text{(if $k = 1$)} \\
\left(\frac{3k+4}{2}, \frac{k+2}{2}, \frac{k+2}{2}, \frac{k-2}{2}\right) \oplus \left(\frac{3k+4}{2}, \frac{k+2}{2}, \frac{k}{2}, \frac{k}{2}\right)
\oplus \left(\frac{3k+4}{2}, \frac{k}{2}, \frac{k}{2}, \frac{k-2}{2}\right)^{\oplus 2}
\oplus \left(\frac{3k+4}{2}, \frac{k-2}{2}, \frac{k-2}{2}, \frac{k-2}{2}\right)
& \text{(if $k \geq 2$).}
\end{cases}
\end{align*}
In this case, all weights are dominant, and hence (2) follows.
The remaining statements (I) and (II) follows as Lemma~\ref{pretilting}.
\end{proof}

\begin{proof}[Proof of Theorem~\ref{thm 2nd tilting}]
Let us show that $\mcT_-^{\sharp}$ is tilting.
Then the statement for $\mcT_+^{\flat}$ follows by choosing a (non-canonical) identification $Z_+ \simeq Z_-$
that induces an isomorphism $X_+ \simeq X_-$.

First, let us check that $\mcT_-^{\sharp}$ is a generator.
It is enough to show that the triangulated subcategory $\langle \mcT_-^{\sharp} \rangle \subset \Db(\coh X_-)$ generated by $\mcT_-^{\sharp}$ contains a generator.
Actually, this category contains the bundle $\pi_-^*\mcU'$, where
\[ \mcU' = \mcO(-3) \oplus \mcO(-2) \oplus \mcS_-(-1) \oplus \mcO(-1) \oplus \mcO \oplus \mcS_-^{\vee} \oplus \mcO(1) \oplus \mcO(2). \]
Since $\mcU'$ is the direct sum of a full exceptional collection (\ref{2nd exc coll}) in $\Db(\coh Z_-)$,
it is a generator of $\D(\Qcoh Z_-)$.
Therefor $\pi_-^* \mcU'$ is a generator of $\D(\Qcoh X_-)$ by \cite[Lemma~3.1]{Hara21}, and hence $\mcT^{\sharp}_{-}$ is also a generator.

Next, let us show that $\mcT_-^{\sharp}$ is pretilting.
Together with the vanishings contained in the statement of Theorem~\ref{thm 1st tilting}, 
it is enough to show that the following $\Ext$ groups are all zero.
\begin{enumerate}
\item[(1)] $\bigoplus_{i \geq 1} \Ext_{X_-}^{i}(\mcP_-^{\vee}(-1), \mcO_{X_-}(2)) \simeq \bigoplus_{i \geq 1} \Ext_{X_-}^{i}(\mcO_{X_-}(-3), \mcP_-)$.
\item[(2)] $\bigoplus_{i \geq 1} \Ext_{X_-}^{i}(\mcO_{X_-}(2), \mcP_-^{\vee}(-1)) \simeq \bigoplus_{i \geq 1} \Ext_{X_-}^{i}(\mcP_-, \mcO_{X_-}(-3))$.
\item[(3)] $\bigoplus_{i \geq 1} \Ext_{X_-}^{i}(\mcP_-^{\vee}(-1), \mcP_-)$.
\item[(4)] $\bigoplus_{i \geq 1} \Ext_{X_-}^{i}(\mcP_-, \mcP_-^{\vee}(-1))$.
\end{enumerate}

For (1), applying $\Ext^*_{X_-}(\mcO_{X_-}(-3), -)$ to (\ref{def of P}) and then using Lemma~\ref{pretilting} proves the desired vanishing.

For (2), there are isomorphisms
\begin{align*}
&\Ext_{X_-}^{i}(\mcO_{X_-}(2), \mcP_-^{\vee}(-1)) \\
\simeq{} & H^{i}(X_-, \mcP_-^{\vee}(-3)) \\
\simeq{} & H^{i}\left(X_-, R(\phi_-)_*\left(\mcE_{\wX}^{\vee}(0,-3)\right)\right) ~\text{~ (by Prop.~\ref{prop exchange}~(3) and the projection formula)} \\
\simeq{} & H^{i}\left(\wX, \mcE_{\wX}^{\vee}(0,-3)\right).
\end{align*}
The bundle $\mcE_{\wX}^{\vee}(0,-3)$ lies in an exact sequence
\begin{align} \label{send tilt ex 1}
0 \to \phi_+^*\mcS_{X_+}(3,0) \otimes \mcO_{\wX}(3E) \to \mcE_{\wX}^{\vee}(0,-3) \to \mcO_{\wX}(5,0) \otimes \mcO_{\wX}(4E) \to 0.
\end{align}
Now using Lemma~\ref{push of exc div} gives
\begin{align*}
H^i(\wX, \phi_+^*\mcS_{X_+}(3,0) \otimes \mcO_{\wX}(3E) ) &\simeq H^i(X_+, \mcS_{X_+}(3) \otimes R(\phi_+)_*\mcO_{\wX}(3E)) \\
&\simeq H^i(X_+, \mcS_{X_+}(3)),
\end{align*}
and this is zero for all $i > 0$ by Lemma~\ref{pretilting}~(2).
Similarly, one has
\begin{align*}
H^i(\wX, \mcO_{\wX}(5,0) \otimes \mcO_{\wX}(4E)) &\simeq H^i(X_+, \mcO_{X_+}(5) \otimes R(\phi_+)_*\mcO_{\wX}(4E)), 
\end{align*}
and again by Lemma~\ref{push of exc div}, the complex $\mcO_{X_+}(5) \otimes R(\phi_+)_*\mcO_{\wX}(4E)$ fits in an exact triangle
\[ \mcO_{X_{+}}(5) \to \mcO_{X_+}(5) \otimes R(\phi_{+})_* \mcO_{\wX}(4E) \to \mcO_{Z_{+}^0}(-1)[-3] \to \mcO_{X_{+}}(5)[1]. \]
Since $H^i(X_+, \mcO_{X_{+}}(5)) = 0$ for all $i > 0$ and $\RG(X_+, \mcO_{Z_{+}^0}(-1)) = 0$, it follows that
$H^i(X_+, \mcO_{X_+}(5) \otimes R(\phi_+)_*\mcO_{\wX}(4E)) = 0$ for all $i > 0$.
Therefore, the sequence (\ref{send tilt ex 1}) yields that 
\begin{align*}
\Ext_{X_-}^{i}(\mcO_{X_-}(2), \mcP_-^{\vee}(-1)) \simeq H^{i}\left(\wX, \mcE_{\wX}^{\vee}(0,-3)\right) = 0
\end{align*}
for all $i > 0$.

For (3), consider the following two exact sequences
\begin{center}
$0 \to \mcS_{X_-}(-1) \to \mcP_-^{\vee}(-1) \to \mcO_{X_-}(1) \to 0$ and 
$0 \to \mcO_{X_-}(-2) \to \mcP_- \to \mcS_{X_-}^{\vee} \to 0$.
\end{center}
Since Lemma~\ref{pretilting} and Lemma~\ref{BBW2} show that
\begin{enumerate}
\item[$\bullet$] $\Ext_{X_-}^i(\mcO_{X_-}(1), \mcO_{X_-}(-2)) = 0$ for all $i>0$,
\item[$\bullet$] $\Ext_{X_-}^i(\mcO_{X_-}(1), \mcS_{X_-}^{\vee}) = 0$ for all $i > 0$,
\item[$\bullet$] $\Ext_{X_-}^i(\mcS_{X_-}(-1), \mcO_{X_-}(-2)) = 0$ for all $i > 0$, and
\item[$\bullet$] $\Ext_{X_-}^i(\mcS_{X_-}(-1), \mcS_{X_-}^{\vee}) = 0$ for all $i > 0$, 
\end{enumerate}
it holds that
\begin{enumerate}
\item[$\diamond$] $\Ext_{X_-}^i(\mcO_{X_-}(1), \mcP_-) = 0$ for all $i > 0$ and
\item[$\diamond$] $\Ext_{X_-}^i(\mcS_{X_-}(-1), \mcP_-) = 0$ for all $i > 0$.
\end{enumerate}
Therefore $\Ext_{X_-}^{i}(\mcP_-^{\vee}(-1), \mcP_-) = 0$ for all $i > 0$, as desired.

Let us prove (4). Consider the same exact sequences as (3).
Since, by Lemma~\ref{pretilting} and Lemma~\ref{BBW2},
\begin{enumerate}
\item[$\bullet$] $\Ext_{X_-}^i(\mcO_{X_-}(-2), \mcO_{X_-}(1)) = 0$ for all $i > 0$,
\item[$\bullet$] $\Ext_{X_-}^i(\mcO_{X_-}(-2), \mcS_{X_-}(-1)) = 0$ for all $i > 0$,
\item[$\bullet$] $\Ext_{X_-}^i(\mcS_{X_-}^{\vee}, \mcO_{X_-}(1)) = 0$ for all $i > 0$, and
\item[$\bullet$] $\Ext_{X_-}^i(\mcS_{X_-}^{\vee}, \mcS_{X_-}(-1)) = 0$ for all $i > 1$,
\end{enumerate}
it follows that  
\begin{enumerate}
\item[$\diamond$] $\Ext_{X_-}^i(\mcO_{X_-}(-2), \mcP_-^{\vee}(-1)) = 0$ for all $i > 0$ and
\item[$\diamond$] $\Ext_{X_-}^i(\mcS_{X_-}^{\vee}, \mcP_-^{\vee}(-1)) = 0$ for all $i > 1$,
\end{enumerate}
and therefore $\Ext_{X_-}^{i}(\mcP_-, \mcP_-^{\vee}(-1)) = 0$ for all $i > 1$.
Thus it remains to show that $\Ext_{X_-}^{1}(\mcP_-, \mcP_-^{\vee}(-1))$ is zero.
This can be shown as
\begin{align*}
\Ext_{X_-}^{1}(\mcP_-, \mcP_-^{\vee}(-1)) &\simeq \Ext_{U}^{1}(j_-^*\mcP_-, j_-^*\mcP_-^{\vee}(-1)) &\text{(by Lemma~\ref{loc coho prop})} \\
&\simeq \Ext_{U}^{1}(j_+^*\mcP_+, j_+^*\mcP_+^{\vee}(1)) &\text{(by Lemma~\ref{compare bundles})} \\
&\simeq \Ext_{X_+}^{1}(\mcP_+, \mcP_+^{\vee}(1)) &\text{(by Lemma~\ref{loc coho prop})} \\
&= 0.  &\text{(by Theorem~\ref{thm 1st tilting})}
\end{align*}
This completes the proof.
\end{proof}

\subsection{Derived equivalence} \label{sect: D-eq}

\begin{prop} \label{tilting exchange}
There are isomorphisms
\begin{center}
$(f_+)_*\mcT^{\sharp}_+ \simeq (f_-)_*\mcT_-^{\sharp}$ and $(f_+)_*\mcT_+^{\flat} \simeq (f_-)_*\mcT_-^{\flat}$
\end{center}
of $R$-modules.
\end{prop}

\begin{proof}
%Let $\refl(Y)$ denote the category of reflexive coherent sheaves over a normal scheme $Y$.
Note that there is an open immersion $\bar{j} \colon U \to \Spec R$ such that $\bar{j} = f_{\pm} \circ j_{\pm}$.
Since $f_{\pm} \colon X_{\pm} \to \Spec R$ are flopping contractions, they are isomorphic in codimension one,
and hence the push-forwards $(f_{\pm})_*$ give equivalences of categories that commute the following diagram.
\[
\begin{tikzcd}
\refl(X_+) \arrow[r, "(f_+)_*"', "\simeq"] \arrow[dr, "j_+^*"', "\simeq"] & \refl(R) \arrow[d, "\bar{j}^*"', "\simeq"] & \refl(X_-)  \arrow[l, "(f_-)_*", "\simeq"'] \arrow[dl, "j_-^*", "\simeq"'] \\
& \refl(U) &
\end{tikzcd}
\]
Now the result follows from Lemma~\ref{compare bundles}.
\end{proof}

\begin{cor} \label{NCCR alg isom}
There are natural isomorphisms of $R$-algebras
\begin{align*}
&\Lambda^{\sharp} \coloneqq \End_{X_+}(\mcT_+^{\sharp}) \simeq \End_R((f_+)_*\mcT_+^{\sharp}) \simeq \End_R((f_-)_*\mcT_-^{\sharp}) \simeq \End_{X_-}(\mcT_-^{\sharp}), ~ \text{and} \\
&\Lambda^{\flat} \coloneqq \End_{X_+}(\mcT_+^{\flat}) \simeq \End_R((f_+)_*\mcT_+^{\flat}) \simeq \End_R((f_-)_*\mcT_-^{\flat}) \simeq \End_{X_-}(\mcT_-^{\flat}).
\end{align*}
\end{cor}

\begin{thm}
There are equivalences of $R$-linear triangulated categories
\begin{align*}
\Phi^{\sharp} &\colon \Db(\coh X_+) \xrightarrow{\RHom(\mcT_+^{\sharp},-)} \Db(\modu \Lambda^{\sharp}) \xrightarrow{\mcT_-^{\sharp} \otimes -} \Db(\coh X_-), \text{~ and} \\
\Phi^{\flat} &\colon \Db(\coh X_+) \xrightarrow{\RHom(\mcT_+^{\flat},-)} \Db(\modu \Lambda^{\flat}) \xrightarrow{\mcT_-^{\flat} \otimes -} \Db(\coh X_-).
\end{align*}
Both $\Phi^{\sharp}$ and $\Phi^{\flat}$ send $\mcO_{X_+}(a)$ to $\mcO_{X_-}(-a)$, where $-2 \leq a \leq 3$.
\end{thm}

\begin{proof}
This theorem is a consequence of Theorem~\ref{thm 1st tilting}, Theorem~\ref{thm 2nd tilting}, and Corollary~\ref{NCCR alg isom}.
\end{proof}

\begin{proof}[Proof of Theorem~\ref{main thm}]
Let $\star \in \{\sharp, \flat\}$.
Then the equivalence $\Phi^{\star}$ satisfies
\begin{align*}
R(f_-)_* \circ \Phi^{\star} &\simeq \RHom_{X_-}(\mcO_{X_-}, \Phi^{\star}(-)) \\
&{}\simeq \RHom_{X_-}(\Phi^{\star}(\mcO_{X_+}), \Phi^{\star}(-)) \\
&{}\simeq \RHom_{X_+}(\mcO_{X_+},-) \\
&{}\simeq R(f_+)_*,
\end{align*}
as desired.
\end{proof}

\begin{rem}
The restriction of $\mcT_{\pm}^{\sharp}$ to $Z_{\pm}^0$ is  isomorphic to the direct sum (with multiplicities) of a full exceptional collections over $Z_{\pm}$, for each $+$ and $-$.
In contrast to tilting bundles for the $G_2^{\dagger}$ type flop constructed in \cite{Hara24},
those exceptional collections are not identical.
They can be visualised as follows. 
The symbol $\ast$ over a bundle indicates that, 
in the construction of a tilting bundle over $X_{\pm}$, 
it will be replaced by another bundle, 
using the $\Ext^1$ described by the dotted arrow.
\[
\begin{tikzpicture}
\draw (-1,0) node[anchor=mid west] {$Z_+$ :};
\draw (0,0) node[anchor=mid west] (a) {$\mcO(-2)$};
\draw ($(a.mid east)+(0.3,0)$) node[anchor= mid west] (b) {$\mcO(-1)$};
\draw ($(b.mid east)+(0.3,0)$) node[anchor=mid west] (c) {$\mcO$};
\draw ($(c.mid east)+(0.3,0)$) node[anchor=mid west] (d) {$\mcS_+^{\vee}$};
\draw ($(d.mid east)+(0.3,0)$) node[anchor=mid west] (e) {$\mcS_+(1)$};
\draw ($(e.mid east)+(0.3,0)$) node[anchor=mid west] (f) {$\mcO(1)$};
\draw ($(f.mid east)+(0.3,0)$) node[anchor=mid west] (g) {$\mcO(2)$};
\draw ($(g.mid east)+(0.3,0)$) node[anchor=mid west] (h) {$\mcO(3)$};
\draw[dashed, ->] (node cs:name=h, anchor=south) to[bend left=20] (node cs:name=e, anchor=south);
\draw[dashed, ->] (node cs:name=d, anchor=south) to[bend left=20] (node cs:name=a, anchor=south);
\draw (node cs:name=e, anchor=north) node{$\ast$};
\draw (node cs:name=d, anchor=north) node{$\ast$};

\draw (-1,-1.5) node[anchor=mid west] {$Z_-$ :};
\draw (0,-1.5) node[anchor=mid west] (A) {$\mcO(-3)$};
\draw ($(A.mid east)+(0.3,0)$) node[anchor= mid west] (B) {$\mcO(-2)$};
\draw ($(B.mid east)+(0.3,0)$) node[anchor=mid west] (C) {$\mcS_-(-1)$};
\draw ($(C.mid east)+(0.3,0)$) node[anchor=mid west] (D) {$\mcO(-1)$};
\draw ($(D.mid east)+(0.3,0)$) node[anchor=mid west] (E) {$\mcO$};
\draw ($(E.mid east)+(0.3,0)$) node[anchor=mid west] (F) {$\mcS_-^{\vee}$};
\draw ($(F.mid east)+(0.3,0)$) node[anchor=mid west] (G) {$\mcO(1)$};
\draw ($(G.mid east)+(0.3,0)$) node[anchor=mid west] (H) {$\mcO(2)$};
\draw[dashed, ->] (node cs:name=F, anchor=south) to[bend left=20] (node cs:name=B, anchor=south);
\draw[dashed, ->] (node cs:name=G, anchor=south) to[bend left=20] (node cs:name=C, anchor=south);
\draw (node cs:name=F, anchor=north) node{$\ast$};
\draw (node cs:name=C, anchor=north) node{$\ast$};
\end{tikzpicture}
\]
Note that, this collection over $Z_-$ gives more $\Ext^1$s over  $X_-$ than those are described by dotted arrows (e.g.~Lemma~\ref{BBW2}). 
Thus it would be little surprising that only two extensions could kill all other extensions and were enough to give a tilting bundle.
\end{rem}

\begin{cor}
Let $R$ be the affine coordinate ring of the affine cone over the polarized manifold $(\OG(3,8), \mcO(1,1))$,
and $\widehat{R}$ the completion of $R$ at the origin.
Then 
\begin{enumerate}
\item[\rm (1)] $\Spec R$ (resp.~$\Spec \widehat{R}$) has exactly two crepant resolutions, and
\item[\rm (2)] $R$ (resp.~$\widehat{R})$ admits an NCCR that is given by a Cohen-Macaulay module and is derived equivalent to both crepant resolutions of $\Spec R$ (resp.~$\Spec \widehat{R}$).
\end{enumerate}
\end{cor}

\section{Application to K3 surfaces}

\subsection{Derived factorization categories}
\begin{setup} \label{setup1}
Let $Z$ be a smooth variety, $\mcE$ a locally free sheaf of finite rank, $s \in H^0(Z, \mcE)$ a regular section,
and $V = V(s)$ its vanishing locus.
The section $s$ naturally associates the function $Q_s \colon X \coloneqq \Tot_Z(\mcE^{\vee}) \to \A^1$.
Let the group $\Gm$ act on $Z$ trivially and on $X$ as a scaling of the fiber.
In other words, for the identity character $\chi_{\id} = \id \colon \Gm \to \Gm$, $X$ is identified with $\Tot_Z(\mcE^{\vee}(\chi_{\id}))$.

Let $p \colon \Tot_V(\mcE^{\vee}|_V) \to V$ be the natural projection and $i \colon \Tot_V(\mcE^{\vee}|_V) \hookrightarrow \Tot_Z(\mcE^{\vee})$ the natural inclusion.
\end{setup}
Under this setup, the data $(X, \chi_{\id}, Q_s)^{\Gm}$ is an example of a gauged Landau-Ginzburg (= LG) model,
and one can consider the associated \textit{derived factorization category}
\[ \Dcoh_{\Gm}(X, \chi_{\id}, Q_s), \]
as in \cite{Hirano17Knorrer}, which satisfies the following theorem.

\begin{thm}[{\cite{Isik13, Shipman12,Hirano17Knorrer}}] \label{hirano1}
Under Setup~\ref{setup1}, the functor
\begin{align} 
i_*p^* \colon \Db(\coh V) \xrightarrow{\sim} \Dcoh_{\Gm}(X, \chi_{\id}, Q_s) 
\end{align}
is an equivalence of $\C$-linear triangulated categories.
\end{thm}

\begin{setup} \label{setup2}
Following Setup~\ref{setup1},
put $S \coloneqq H^0(X, \mcO_X)$, and let $f \colon X \to \Spec S$ be the affinization morphism.
Note that $S$ naturally admits a $\Gm$-action, and $f$ is equivariant with respect to those actions.
Let $Q \colon \Spec S \to \A^1$ be the function such that $Q_s = f \circ Q$.

Assume that $f$ is projective, and $\mcT$ is a tilting bundle over $X$ that is $\Gm$-equivariant as a sheaf over $X$.
Then the endomorphism algebra $\Lambda \coloneqq \End_X(\mcT)$ is a module-finite $S$-algebra admitting a $\Gm$-equivariant structure.
\end{setup}

Under this setup, one can consider the derived factorization category
\[ \Dmod_{\Gm}(\Lambda, \chi_{\id}, Q) \]
for the noncommutative gauged LG model $(\Lambda, \chi_{\id}, Q)^{\Gm}$ (c.f.~\cite{Hirano21}).
Then the following holds.

\begin{thm}[{\cite{Hirano21}}] \label{hirano2}
Under Setup~\ref{setup2}, the functor
\begin{align} 
\RHom(\mcT, -) \colon \Dcoh_{\Gm}(X, \chi_{\id}, Q_s) \xrightarrow{\sim}  \Dmod_{\Gm}(\Lambda, \chi_{\id}, Q).
\end{align}
is an equivalence of categories.
\end{thm}

\begin{rem}
\begin{enumerate}
\item[(1)] For the construction of the functor $\RHom(\mcT,-)$, see \cite[Section~4]{Hirano21}.
\item[(2)] The main result of \cite{OT21} is relevant to the result above, but it assumes one extra condition that the tilting bundle $\mcT$ is obtained by the pull-back of a bundle over the base $Z$.
Our tilting bundles in Theorem~\ref{thm 1st tilting} and \ref{thm 2nd tilting} do not satisfy this assumption. 
\end{enumerate}
\end{rem}

\subsection{Application to the canonical local models}

Let $(Z_+, \mcE_+)$ and $(Z_-, \mcE_-)$ be simple Mukai pairs such that give two distinct projective space bundle structures of the roof $W \simeq \P_{Z_+}(\mcE_+) \simeq \P_{Z_-}(\mcE_-)$,
and
\[ \begin{tikzcd} 
X_+ \coloneqq \Tot(\mcE_+^{\vee}) \arrow[dr, "f_+"] \arrow[rr, dashrightarrow] & & X_- \coloneqq \Tot(\mcE_-^{\vee}) \arrow[dl,"f_-"'] \\
& \Spec R &
\end{tikzcd} \]
the associated canonical local model of the simple flop of the corresponding type.
Consider the actions of $\Gm$ on $X_{\pm}$ and $\Spec R$ as in Setup~\ref{setup1} and \ref{setup2}.
The images of the open embeddings $j_{\pm} \colon U \to X_{\pm}$ are $\Gm$-invariant 
and hence the common open subset $U$ also admits a $\Gm$-action.

\begin{thm} \label{Hirano equiv for CY}
Assume that there are $\Gm$-equivariant tilting bundles $\mcT_{\pm}$ over $X_{\pm}$
such that $j_+^*\mcT_+ \simeq j_-^*\mcT_-$ as $\Gm$-equivariant bundles.
Let $s_+ \in H^0(\mcE_+)$ be a section, and $s_- \in H^0(\mcE_-)$ the corresponding section under the identification $H^0(\mcE_+) \simeq H^0(\mcE_-)$.
If $s_{\pm}$ are regular, then there is an exact equivalence of categories
\[ \Db(\coh V(s_+)) \simeq \Db(\coh V(s_-)). \]
\end{thm}

\begin{proof}
Let $Q_{\pm} \colon X_{\pm} \to \A^1$ and $Q \colon \Spec R \to \A^1$ be the corresponding functions.
Put $\Lambda_{\pm} \coloneqq \End_{X_{\pm}}(\mcT_{\pm})$.
Then Theorem~\ref{hirano1} and \ref{hirano2} show that there are natural equivalences
\[ \Db(\coh V(s_{\pm})) \simeq \Dcoh_{\Gm}(X_{\pm}, \chi_{\id}, Q_{\pm}) \simeq \Dmod_{\Gm}(\Lambda_{\pm}, \chi_{\id}, Q) \]
for each $+$ and $-$.
Then the assumption yields a $\Gm$-equivariant isomorphism $\Lambda_+ \simeq \Lambda_-$ of $R$-algebras,
and hence the composite of all the above equivalences gives the desired equivalence $\Db(\coh V(s_+)) \simeq \Db(\coh V(s_-))$.
\end{proof}

\subsection{K3 surfaces of degree $12$ and the roof of type $D_4$}

Let us consider the canonical local model of the simple flop of type $D_4$.
The projection $\pi \colon \wX \to W = \OG(3,8)$ is $\Gm$-equivariant, where $\Gm$ acts on $\wX$ by the fiber-wise scaling
and on $W$ trivially.
The birational contractions $\phi_{\pm} \colon \wX \to X_{\pm}$ are also $\Gm$-equivariant.

\begin{prop} \label{roof gives tilting}
For $\star \in \{\sharp, \flat\}$, there exist $\Gm$-equivariant structures on $\mcT_+^{\star}$ and $\mcT_-^{\star}$
such that $j_+^*\mcT_+^{\star} \simeq j_-^*\mcT_-^{\star}$ as $\Gm$-equivariant bundles.
\end{prop}

\begin{proof}
Consider the bundle
\[ \mcG^{\sharp} = \mcO(-2,0) \oplus \mcO(-1,0) \oplus \mcO \oplus \mcE(-1,-1) \oplus \mcE^{\vee}(0,-1) \oplus \mcO(0,-1) \oplus \mcO(0,-2) \oplus \mcO(0,-3), \]
over the roof $W = \OG(3,8)$, where $\mcE$ is the key bundle.
Note that
\begin{align*}
(\phi_+)_*\mcO_{\wX}(-k,0) &\simeq \mcO_{X_+}(-k), \\
(\phi_-)_*\mcO_{\wX}(-k,0) &\simeq (\phi_-)_*\left( \mcO_{\wX}(0,k) \otimes O_{\wX}(kE) \right) \simeq \mcO_{X_-}(k)  \\
(\phi_+)_*\mcO_{\wX}(0,-k) &\simeq (\phi_+)_* \left( \mcO_{\wX}(k,0) \otimes O_{\wX}(kE) \right) \simeq \mcO_{X_+}(k), ~\text{and} \\
(\phi_-)_*\mcO_{\wX}(0,-k) &\simeq \mcO_{X_-}(-k)
\end{align*}
for all $0 \leq k \leq 3$ by Lemma~\ref{push of exc div}.
Thus together with Proposition~\ref{prop exchange}, there exists an isomorphism
$(\phi_{\pm})_*\pi^*\mcG^{\star} \simeq \mcT_{\pm}^{\star}$ for each $+$ and $-$.
Since $\Gm$ trivially action on $W = \OG(3,8)$ and $\pi \colon \wX \to W$ is $\Gm$-equivariant,
the pull-back $\pi^*\mcG^{\sharp}$ admits a natural $\Gm$-equivariant structure.
Thus tilting bundles $\mcT_{\pm}^{\sharp} \simeq (\phi_{\pm})_*\pi^*\mcG^{\sharp}$ also admit $\Gm$-equivariant structures such that $j_+^*\mcT_+^{\sharp} \simeq j_-^*\mcT_-^{\sharp}$ as $\Gm$-equivariant bundles.

Similarly,  put
\[ \mcG^{\flat} = \mcO(-3,0) \oplus \mcO(-2,0) \oplus \mcE^{\vee}(-1,0) \oplus \mcO(-1,0) \oplus \mcO \oplus \mcE(-1,-1) \oplus \mcO(0,-1) \oplus \mcO(0,-2). \]
Then
\[ \text{$(\phi_+)_*\pi^*\mcG^{\flat} \otimes \mcO_{X_+}(1) \simeq \mcT^{\flat}_+$ and 
$(\phi_-)_*\pi^*\mcG^{\flat} \otimes \mcO_{X_-}(-1) \simeq \mcT^{\flat}_-$.} \]
Thus, as before,
tilting bundles $\mcT_{\pm}^{\flat}$ admit $\Gm$-equivariant structures 
such that $j_+^*\mcT_+^{\flat} \simeq j_-^*\mcT_-^{\flat}$ as $\Gm$-equivariant bundles.
\end{proof}

\begin{proof}[Proof of Corollary~\ref{main cor}]
This is a consequence of Corollary~\ref{roof gives tilting} and Theorem~\ref{Hirano equiv for CY}.
\end{proof}

\begin{rem} \label{remark for K3}
The following are known for the (smooth) K3 surfaces that appear as the zero-locus of a regular section $s_+ \in H^0(\mcS_+(2))$.
\begin{enumerate}
\item[$\bullet$] By \cite{Mukai88} and \cite[Proposition~3.4]{IMOU20}, a general K3 surface of degree $12$ appears as the zero-locus of a section $s_+ \in H^0(\mcS(2))$.
\item[$\bullet$] Bby \cite[Corollary~4.10]{KR22}, if $s_+$ is very general, then $V(s_+)$ and $V(s_-)$ are non-isomorphic pair of K3 surfaces.
\end{enumerate}
The derived equivalence between $V(s_+)$ and $V(s_-)$ was first proved by \cite[Proposition~3.4]{IMOU20} using projective duality, 
and then alternative proofs were given by \cite[Corollary~3.3]{KR22} using lattice theory and 
by \cite[Section~6.2.1]{Xie24} using mutations of semiorthogonal decompositions.
Note that, in contrast to Corollary~\ref{main cor}, the proofs in \cite{KR22, Xie24} applies only to smooth K3 surfaces.

On the other hand, in \cite{IMOU20}, the derived equivalence is proved using a moduli-theoretic description of $V(s_-)$. This approach also applies only in the smooth case. 
However, the proofs of the key propositions \cite[Propositions~3.4 and 4.1]{IMOU20} remain valid for arbitrary regular sections $s_{\pm} \in H^0(\mcS_{\pm}(2))$. 
Hence, together with the homological projective duality theorem in \cite[Section~6.2]{Kuz06}, 
one can deduce a derived equivalence
\[
\Db(\coh V(s_+)) \simeq \Db(\coh V(s_-))
\]
even in the presence of singularities.
This derived equivalence, however, factors through several non-canonical identifications, 
such as an isomorphism 
\[
V(s_+) \simeq \mathrm{OG}_-(5,10) \cap L
\]
for some $7$-dimensional linear subspace $L \subset \mathbb{P}^{15}$.

In this sense, our construction of the equivalences in Corollary~\ref{main cor} provides a more canonical and uniform treatment for the family 
\[
\left\{ \left(V(s_+), V(s_-) \right) \mid \text{$s_{\pm} \in H^0(\mathcal{S}_{\pm}(2))$ are regular} \right\}
\]
of pairs of (possibly singular) K3 surfaces.
\end{rem}

\begin{rem}
The same statement as in Proposition~\ref{roof gives tilting} holds for tilting bundles constructed in \cite{Hara17, Seg16, Hara22, Hara21, Hara24, DHKR25}.
In particular, a similar statement as in Corollary~\ref{main cor} holds for the canonical local models of types $A_4^G$, $C_2$, $G_2$, and $G_2^{\dagger}$.

As explained in Remark~\ref{rem roof K3}, the roofs with which K3 surfaces are relevant are exactly of type $D_4$ or $G_2^{\dagger}$.
The corresponding simple Mukai pairs are $(\Q^6, \mcS(2))$ and $(\Q^5, \mcG(1))$, respectively, 
where $\mcS$ is a spinor bundle and $\mcG$ is an Ottaviani bundle.
In both cases, the degree of K3 surfaces is $12$.
However, in contrast to the spinor bundle, Ottaviani bundles over $\Q^5$ have moduli.
This makes the study of the type $G_2^{\dagger}$ case more involved than the type $D_4$ case.
See \cite{KR22}.
\end{rem}

\appendix
\section{Simple flops and derived equivalences} \label{appendix}

This appendix summarises the existing progresses of the study of derived categories for (known) simple flops in Table~\ref{table:simple flop} below.
Note that simple flops have not been classified yet, and non-homogeneous examples such as type $G_2^{\dagger}$ flop may exist more.
The complete classification of simple flops is an important open problem from the perspectives of birational geometry and the studies of Fano manifolds and projective Calabi-Yau manifolds, besides the theory of derived categories.
In the table, each row shows the following information and references.
\begin{enumerate}
\item[$\bullet$] Dim of flop: the dimension of the canonical local model of the simple flop.
\item[$\bullet$] Dim of CY: the dimension of the associated projective Calabi-Yau manifolds. Here we put $\dim \emptyset = -1$.
\item[$\bullet$] D-equivalence: the first paper that proves the derived equivalence for that type of the simple flop.
\item[$\bullet$] FM kernel: the paper that constructs the explicit Fourier-Mukai kernel.
\item[$\bullet$] Tilting equiv: the paper that constructs tilting bundles for the canonical local models and applies those bundles to prove the derived equivalence.
\end{enumerate}

The following provides several comments on the table.

\begin{enumerate}
\item[(1)] Note that the existence of a Fourier-Mukai kernel follows from the general theory of dg categories once a derived equivalence is established in reasonable ways such as the one using tilting bundles or the one using mutations of semiorthogonal decompositions.
However, it is often difficult to give an explicit description of such a kernel.
In all the cases where an explicit Fourier-Mukai kernel is known (namely, of types $A_n \times A_n$, $A_n^M$, and $C_2$), that kernel is given by the structure sheaf of the fiber product $X_+ \times_{\Spec R} X_-$.
It would be interesting to ask if the same construction can give a Fourier-Mukai kernel for other types of simple flops.
\item[(2)] Although it is not included in the table, the derived equivalence for the type $D_5$ flop has already been addressed by \cite{RX24}.
However, an explicit description of a Fourier-Mukai kernel and the construction of an equivalence via tilting bundles for this type remain open problems.
\item[(3)] As the first proof of the derived equivalence for the simple flop of type $C_2$, the work of Morimura \cite{Morimura22} is also listed with the original work by Segal \cite{Seg16}, since Morimura's proof that is based on semiorthogonal decompositions and their mutations can be applied in more general settings beyond the canonical local model.
\end{enumerate}

\begin{landscape}
\begin{table}[h]
  \centering
  \begingroup
  \renewcommand{\arraystretch}{1.5}

  \begin{tabular}{|c|c||c|c|c|c|c|c|c|c|} \hline
  \multicolumn{2}{|c||}{Classification} & $A_n \times A_n$ & $A_n^M$ & $A^G_n$ & $C_n$ & $D_n$ & $F_4$ & $G_2$ & $G_2^{\dagger}$ \\ \hline
  
  \multicolumn{2}{|c||}{Condition for $n$} & $n \geq 1$ & $n \geq 2$  & $n$ is even, $n \geq 4$ & $n \equiv 2~\mod~3$ & $n \geq 4$ & \multicolumn{3}{c}{} \\ \hhline{=======~~~}

  \multirow{6}{*}{Smallest rank} & Type & $A_1 \times A_1$ & $A_2^M$ & $A_4^G$ & $C_2$ & $D_4$ & \multicolumn{3}{c}{}    \\ \cline{2-7}
  & Dim of flop & $3$ & $4$ & $9$ & $5$ & $10$ & \multicolumn{3}{c}{}    \\ \cline{2-7}
  & Dim of CY & $-1$ & $0$ & $3$ & $1$ & $2$ & \multicolumn{3}{c}{}    \\ \hhline{|~|======|~~~}
  & D-equivalence & \multicolumn{2}{c|}{\multirow{3}{*}{Same as below}} & \cite{Morimura22} & \cite{Seg16, Morimura22} & \cite{Xie24} & \multicolumn{3}{c}{}  \\ \cline{2-2} \cline{5-7} 
  & FM kernel & \multicolumn{2}{c|}{} & ? & \cite{Hara22} & ? & \multicolumn{3}{c}{}  \\ \cline{2-2} \cline{5-7} 
  & Tilting equiv & \multicolumn{2}{c|}{} & \cite{DHKR25} & \cite{Seg16, Hara22} & This article & \multicolumn{3}{c}{}  \\ \hline\hline

  \multirow{5}{*}{General rank} & Dim of flop & $2n+1$ & $2n$ & $(n+2)^2/4$ & $(2n^2+3n+1)/3$ & $(n^2+n)/2$ & $23$ & $7$ & $8$  \\ \cline{2-10}
  & Dim of CY & $-1$ & $0$ & $(n^2-4)/4$ & $(2n^2-n-3)/3$ & $(n^2-3n)/2$ & $17$ & $3$ & $2$  \\ \hhline{|~|=========|}
  & D-equivalence & \multirow{2}{*}{\cite{BO02}} & \multirow{2}{*}{\cite{Kaw02, Nam03}} & \multirow{3}{*}{?} & \multirow{3}{*}{?} & \multirow{3}{*}{?} & \multirow{3}{*}{?} & \cite{Ueda19} & \cite{Hara24, Xie24}  \\ \cline{2-2} \cline{9-10}
  & FM kernel & & &  &  &  & & ? & ?  \\ \cline{2-4} \cline{9-10}
  & Tilting equiv & \multicolumn{2}{c|}{well-known (c.f.~\cite{Hara17})} &  &  &  &  & \cite{Hara21} & \cite{Hara24}  \\ \hline
  \end{tabular}

  \endgroup
\vspace{.5em}
  \caption{Simple flops and derived equivalences}
  \label{table:simple flop}
\end{table}
\end{landscape}

\bibliography{flop}
\bibliographystyle{amsalpha}

\end{document}